\documentclass{article}
\input{macros-arXiv.sty}

\usepackage{url}            %
\usepackage{booktabs}       %
\usepackage{amsfonts, bm, amssymb}       %
\usepackage{nicefrac}       %
\usepackage{microtype}      %
\usepackage{tikz} 
\usepackage{float}
\usepackage{booktabs}

\usepackage{array}
\usepackage{diagbox}
\usepackage{mathtools}
\usepackage{multirow, makecell}
\usepackage{rotating}
\usepackage{graphicx}
\usepackage{mathdots}

\usepackage{caption}

\usepackage{float}
\usepackage{setspace}
\usepackage{fullpage}

\usepackage[bottom]{footmisc}
\usepackage[colorlinks,linkcolor={blue!75!black},urlcolor=red,citecolor=ForestGreen,backref=page]{hyperref}
 \renewcommand*{\backrefalt}[4]{%
    \ifcase #1%
     \or (Cited on page~#2.)%
     \else (Cited on pages~#2.)%
    \fi%
    }

\usepackage{threeparttable}

\usepackage{multirow}
\usepackage{makecell}
\usepackage{arydshln}
\usepackage{tabularx}

\usepackage{arydshln}

\usepackage{enumerate}
\usepackage{enumitem}

\usepackage[sort, numbers]{natbib}

\usepackage{float}
\floatstyle{plain}
\newfloat{floatalgo}{ht}{floatalgo}

\usepackage{adjustbox}

\usepackage{caption}

\usepackage{mathrsfs}

\usepackage{wrapfig}

\usepackage{subfigure}
\usepackage{math_command}
\newcommand{\rbr}[1]{\left(#1\right)}

\usepackage{enumitem}

\title{\bf Two Sides of One Coin: the Limits of Untuned SGD and the Power of Adaptive Methods}

\addtocounter{footnote}{1} %
\author{
Junchi Yang
$^\dagger$
$^*$
\and
Xiang Li
$^\dagger$
$^*$
\and
Ilyas Fatkhullin 
$^\dagger$
\and
Niao He
$^\dagger$
}

\date{\vspace{1ex}}

\begin{document}
\maketitle
\def\thefootnote{$*$}\footnotetext{
Equal contribution.
}
\def\thefootnote{$\dagger$}\footnotetext{Department of Computer Science, ETH Zurich, Switzerland.
Emails: \texttt{junchi.yang@inf.ethz.ch},
\texttt{xiang.li@inf.ethz.ch},
\texttt{ilyas.fn979@gmail.com},
\texttt{niao.he@inf.ethz.ch}.
}
\def\thefootnote{\arabic{footnote}}

\begin{abstract}
The classical analysis of Stochastic Gradient Descent (SGD) with polynomially decaying stepsize $\eta_t = \eta/\sqrt{t}$ relies on well-tuned $\eta$ depending on  problem parameters such as Lipschitz smoothness constant, which is often unknown in practice.  
In this work, we prove that SGD with arbitrary $\eta > 0$, referred to as \textit{untuned SGD}, still attains an order-optimal convergence rate $\widetilde{\cO}(T^{-1/4})$ in terms of gradient norm for minimizing smooth objectives.  
Unfortunately, it comes at the expense of a catastrophic exponential dependence on the smoothness constant, which we show is unavoidable for this scheme even in the noiseless setting. We then examine three families of adaptive methods --- Normalized SGD (NSGD), AMSGrad, and AdaGrad --- unveiling their power in preventing such exponential dependency in  the absence of information about the smoothness parameter and boundedness of stochastic gradients. 
Our results provide  theoretical justification for the advantage of adaptive methods over untuned SGD in alleviating the issue with large gradients.  

\end{abstract}

\section{Introduction}
\label{sec:intro}

In this work, we study the stochastic optimization problem of the form: 
\begin{align*}
  \min_{x \in \bR^{d}}  \ f(x) = \mathbb{E}_{\xi \sim P}\left[ F(x;\xi)\right], %
  \label{eq:objective} 
\end{align*}
where $P$ is an unknown probability distribution, and $f: \bR^{d} \rightarrow \bR$ is an $\ell$-Lipschitz smooth function and can be non-convex. In the context of machine learning, $\xi$ may represent an individual training sample from the data distribution  $P$, and $x$ denotes the weights of the model. 

Stochastic Gradient Descent (SGD), originated from the seminal work~\citep{robbins1951stochastic}, performs the following update iteratively:
\begin{equation*}
  x_{t+1} = x_t - \eta_t \nabla F(x_t; \xi_t),\label{eq:sgd}
\end{equation*}
where $\eta_t > 0$ is some stepsize and $\nabla F(x_t;\xi_t)$ is an unbiased stochastic gradient.  
SGD has shown remarkable empirical success in many modern machine learning applications, e.g., \citep{Bengio09ftml,sutton-barto18}.  Its efficiency is usually attributed to its cheap per iteration cost and the ability to operate in an online fashion, making it suitable for large-scale problems. %
However, empirical evidence also reveals undesirable behaviors of SGD, often related to challenges in selecting appropriate stepsizes. In particular, a number of works report the \textit{gradient explosion} effect \citep{Bengio94_ltdep,pascanu2013difficulty,Goodfellow_2016}  during the initial phase of training, which may eventually lead to divergence or prohibitively slow convergence. 
The phenomenon is also observed in our experiments (see Figure~\ref{subfig:mnist}) when the stepsize is poorly chosen.
Unfortunately, this phenomenon is not well understood from a theoretical point of view. The classical analysis of SGD in the smooth non-convex case \citep{ghadimi2013stochastic}, prescribes to select a non-increasing sequence of stepsizes $\left\{ \eta_t \right\}_{t\geq1}$ with $\eta_1 < 2/\ell$. In particular, the choice $\eta_t = 1 / ( \ell \sqrt{t} )$, guarantees\footnote{Given access to unbiased stochastic gradient oracle with bounded variance.}  to find a point $x$ with $\Ep{\norm{\nabla f(x) }} \leq \epsilon$ after $\cO\left( \epsilon^{-4} \right)$ stochastic gradient calls, which is also known to be unimprovable in the smooth non-convex setting unless additional assumptions are made \citep{arjevani2022lower,drori2020complexity}.

\begin{table*}[t]
		\centering
		\small
\caption{Complexities of finding an $\epsilon$-stationary point for SGD, NSGD \citep{nesterov1984minimization}, NSGD-M \citep{cutkosky2020momentum}, AMSGrad-norm (norm version of AMSGrad \citep{reddi2018convergence}), and AdaGrad-norm \citep{streeter2010less}.
  We only assume $f$ is $\ell$-smooth, and unbiased stochastic gradients have bounded variance $\sigma^2$. Hyper-parameters (e.g., $\gamma$ and $\eta$) are untuned. Here, $\widetilde{\cO}$ and $\Omega$ hide polynomial terms in problem parameters and hyper-parameters. The bounds are with respect to specific algorithms and stepsizes, and lower bounds for general first-order methods still hold \citep{carmon2020lower, arjevani2022lower}.  
  We denote the effective stepsize at iteration $t$ as $\eta_t$.
  }
		\renewcommand{\arraystretch}{1.49}
  \hspace*{-1.2em}
		\begin{threeparttable}[b]
			\begin{tabular}{c | c | c | c | c}
				\hline
				\hline
				\multirow{2}{*}{\textbf{Algorithms}} & 
				\multirow{2}{*}{
					\makecell[c]{
						\textbf{Upper bound;} \vspace{0.1em}\\
						\textbf{deterministic} 
					} 
				} & 
				\multirow{2}{*}{
					\makecell[c]{
						\textbf{Lower bound;} \vspace{0.1em}\\
						\textbf{deterministic}
					}
				} 
				&
                    \multirow{2}{*}{
					\makecell[c]{
						\textbf{Upper bound;} \vspace{0.1em}\\
						\textbf{stochastic}
					}
				} 
				&
                    \multirow{2}{*}{
					\makecell[c]{
						\textbf{Lower bound;} \vspace{0.1em}\\
						\textbf{stochastic}
					}
				} 
				\\
				& &  & &
				\\
				\hline 
				\hline
				\multirow{2}{*}{
                    \makecell[c]{
					    SGD (Alg. \ref{alg:sgd}) \vspace{0.1em}\\
						$\eta_t = \frac{\eta}{\sqrt{t+1}}$
					}
                    }
				& \hspace*{-3mm}
                    \multirow{2}{*}{
					\makecell[c]{
					    $\widetilde{\cO}\left( (4e)^{2(\eta\ell)^{2}}\epsilon^{-4}\right)$ \\
					  $[$Thm.  \ref{thm:unbounded}, \ref{thm:general_stepsize}$]$
					}
				} \hspace*{-3mm}
				& 
                    \hspace{-1mm}\multirow{2}{*}{
					\makecell[c]{
						$\Omega\left( (8e)^{\eta^2\ell^2/8}\epsilon^{-4}\right)$ \vspace{0.1em}\\
					$[$Thm. \ref{thm:lower_bound}$]$
					}
				}\hspace{-2mm} 
				& 
                    \hspace{-2mm}\multirow{2}{*}{
					\makecell[c]{
						$\widetilde{\cO}\left( (4e)^{2(\eta\ell)^{2}}\epsilon^{-4}\right)$  \\
					   $[$Thm. \ref{thm:unbounded}, \ref{thm:general_stepsize}$]$
					}
				}\hspace{-2mm} 
				& 
                    \multirow{2}{*}{
					\makecell[c]{
						$\Omega\left( (8e)^ {\eta^2\ell^2/8}\epsilon^{-4}\right)$ \vspace{0.1em}\\
					$[$Thm. \ref{thm:lower_bound}$]$
					}
				} 
				\\
                    & & & &
                    \\
				\hline
				\multirow{2}{*}{
                    \makecell[c]{
						NSGD (Alg. \ref{alg:nsgd}) \vspace{0.1em}\\
						$\eta_t = \frac{\gamma}{\sqrt{t+1}\| g(x_t; \xi_t)\|}$
					}
                    }
				& 
                    \multirow{2}{*}{
					\makecell[c]{
					    $\widetilde{\cO}\left( \epsilon^{-2}\right)$ \vspace{0.1em}\\
						\citep{cutkosky2020momentum} \&
                        $[$Prop. \ref{thm:NSGD}$]$
					}
				} 
				& 
                    \multirow{2}{*}{
					\makecell[c]{
						$\Omega\left( \epsilon^{-2}\right)$ \vspace{0.1em}\\
						\citep{carmon2020lower}
					}
				} 
				& 
                    \multirow{2}{*}{
                    \makecell[c]{
					N/A due to \\ lower bound
                 }
				} 
				& 
                    \multirow{2}{*}{
					\makecell[c]{
						Nonconvergent \vspace{0.1em}\\
						$[$Thm. \ref{lemma:nonconvergence}$]$
					}
				} 
				\\
                    & & & &
                    \\
				\hline
				\multirow{2}{*}{
                    \makecell[c]{
					    NSGD-M (Alg. \ref{alg:nsgdm}) \vspace{0.1em}\\
						$\eta_t = \frac{\gamma}{(t+1)^{\alpha}\| g_t\|}$
					}
                    }
				& 
                    \multirow{2}{*}{
					\makecell[c]{
					    $\widetilde{\cO}\left( \epsilon^{-2}\right)$, $\alpha = 1/2$ \vspace{0.1em}\\
         \citep{cutkosky2020momentum} \&
						$[$Prop. \ref{thm:NSGD}$]$
					}
				} 
				& 
                    \multirow{2}{*}{
					\makecell[c]{
						$\Omega\left( \epsilon^{-2}\right)$ \vspace{0.1em}\\
						\citep{carmon2020lower}
					}
				} 
				& 
                    \hspace{-1mm}\multirow{2}{*}{
					\makecell[c]{
						$\widetilde{\cO}\left( \epsilon^{-4}\right)$, $\alpha = 3/4$  \vspace{0.1em}\\
      \citep{cutkosky2020momentum} \&
						$[$Prop. \ref{alg:nsgdm}$]$
					}
				} 
				& 
                    \multirow{2}{*}{
					\makecell[c]{
						$\Omega\left( \epsilon^{-4}\right)$ \vspace{0.1em}\\
						\citep{arjevani2022lower}
					}
				}
				\\
                    & & & &
                    \\
				\hline
				\multirow{2}{*}{
                    \makecell[c]{
					   AMSGrad-norm 
                        (Alg. \ref{alg:amsgrad_full})\vspace{0.2em}\\
						 $\eta_t = \frac{\gamma}{\sqrt{(t+1)\hat v^2_{t+1}}}$
					}
                    }
				& 
                    \multirow{2}{*}{
					\makecell[c]{
					    $\widetilde{\cO}\left( \epsilon^{-4}\right)$ \vspace{0.1em}\\
						$[$Thm. \ref{theorem:ams_grad}, \ref{thm:amsgrad_general_ss}$]$
					}
				} 
				& 
                    \multirow{2}{*}{
					\makecell[c]{
						$\Omega\left( \epsilon^{-4}\right)$ \vspace{0.1em}\\
                        $[$Thm. \ref{thm:amsgrad_deter_lower_bound}$]$
					}
				} 
				& 
                    \multirow{2}{*}{
                    \makecell[c]{
					N/A due to \\ lower bound
                 }
				} 
				& 
                    \hspace{-2mm }\multirow{2}{*}{
					\makecell[c]{
						$\Omega\left( \epsilon^{-\frac{2}{1-\zeta}}\right)$ 
					$\forall \zeta\in (0.5, 1)$ \vspace{0.3em}\\	$[$Thm. \ref{thm:amsgrad_stoc}$]$
					}
				} \hspace*{-4mm}
				\\
                    & & & &
                    \\
				\hline
                    \hspace{-3mm}\multirow{2}{*}{
                    \makecell[c]{
					    AdaGrad-norm \vspace{0.1em}
						(Alg. \ref{alg:adagrad}) \\ $\eta_t = \frac{\eta}{\sqrt{ v_0^2 + \sum_{k = 0}^t \|g(x_k;\xi_k)\|^2}}$
					}
                    } \hspace*{-3mm}
				& 
                    \multirow{2}{*}{
					\makecell[c]{
					    $\widetilde{\cO}\left( \epsilon^{-2}\right)$ \vspace{0.1em}\\
         \citep{yang2022nest} \&
						$[$Prop. \ref{prop:adagrad}$]$
					}
				} 
				& 
                    \multirow{2}{*}{
					\makecell[c]{
						$\Omega\left( \epsilon^{-2}\right)$  \vspace{0.1em}\\
						\citep{carmon2020lower}
					}
				} 
				& 
                    \multirow{2}{*}{
                    \makecell[c]{
						$\widetilde{\cO}\left( \epsilon^{-4}\right)$  \vspace{0.1em}\\
      \citep{yang2022nest} \&
						$[$Prop. \ref{prop:adagrad}$]$
					}
				} 
				& 
                    \multirow{2}{*}{
					\makecell[c]{
						$\Omega\left( \epsilon^{-4}\right)$ \vspace{0.1em}\\
					\citep{arjevani2022lower}
					}
				} 
				\\
                    & & & &
                    \\
				\hline
				\hline
			\end{tabular}
		\end{threeparttable}
  \label{table:summary_results1}
 \end{table*}

However, the bound on the smoothness parameter $\ell$ is usually not readily available for practitioners, and the limited computing power usually refrains them from exhaustive tuning to find the best stepsize. It is therefore important to provide theoretical understanding for SGD with an arbitrary stepsize (which we refer to as \emph{untuned SGD}) that is agnostic to the problem parameter. The following intriguing question remains elusive in the stochastic optimization literature:

\begin{quote}
\textit{How does untuned SGD with decaying stepsize $\eta_t = \eta/\sqrt{t}$ perform when $\eta$ is independent of the smoothness parameter? How to explain the undesirably large gradients encountered in training with SGD?}
\end{quote}

Recently, there has been a surge of interest in adaptive gradient methods such as \mbox{Adam} \citep{kingma2015adam}, \mbox{RMSProp} \citep{hinton2012neural}, \mbox{AdaDelta} \citep{Zeiler_2011_Adadelta}, \mbox{AMSGrad} \citep{reddi2018convergence}, \mbox{AdaGrad} \citep{duchi2011adaptive}, Normalized SGD \citep{hazan2015beyond} and many others. These methods automatically adjust their stepsizes based on past stochastic gradients rather than using pre-defined iteration-based schedules. Empirically, they are observed to converge faster than SGD and mitigate the issue of gradient explosion across a range of problems, even without explicit knowledge of problem-specific parameters \citep{kingma2015adam,Liu_OnVarianceAdapt_2020, pascanu2013difficulty}. Figure~\ref{subfig:quadratic} provides a basic illustration of performance differences between SGD with $\eta_t = 1/\sqrt{t}$ stepsizes and adaptive schemes such as AdaGrad and Normalized SGD with momentum (NSGD-M) \citep{cutkosky2020momentum}. 
Notably, when the initial stepsize is too large (compared to $1/\ell$ value), SGD reaches the region with \textit{large gradients}, while adaptive methods do not suffer from such effect. 
However, the theoretical benefits of adaptive methods over SGD remain unclear. A large number of existing analyses of adaptive methods assume bounded gradients, or even stochastic gradients, precluding not only a fair comparison with SGD whose convergence does not need bounded gradient but also the possibility to explain their benefit when facing gradient explosions. While recent developments  show that AdaGrad-type methods \citep{Faw_2022_Power_Adapt, yang2022nest} can attain $\widetilde{\cO}\left( \epsilon^{-4} \right) $ sample complexity under
the same standard assumptions as for SGD analysis,
there still lacks a good explanation for the huge performance gap observed in practice despite SGD with well-tuned stepsizes theoretically achieving the lower complexity bound. We will address the following open question:

\begin{quote}
\textit{Can we justify the theoretical benefits of adaptive methods over untuned SGD for smooth non-convex problems without assuming bounded gradients?} 
\end{quote} 

\begin{figure}[t]
    \centering
    \subfigure[quadratic function]{
      \centering
      \includegraphics[width=0.45\textwidth]{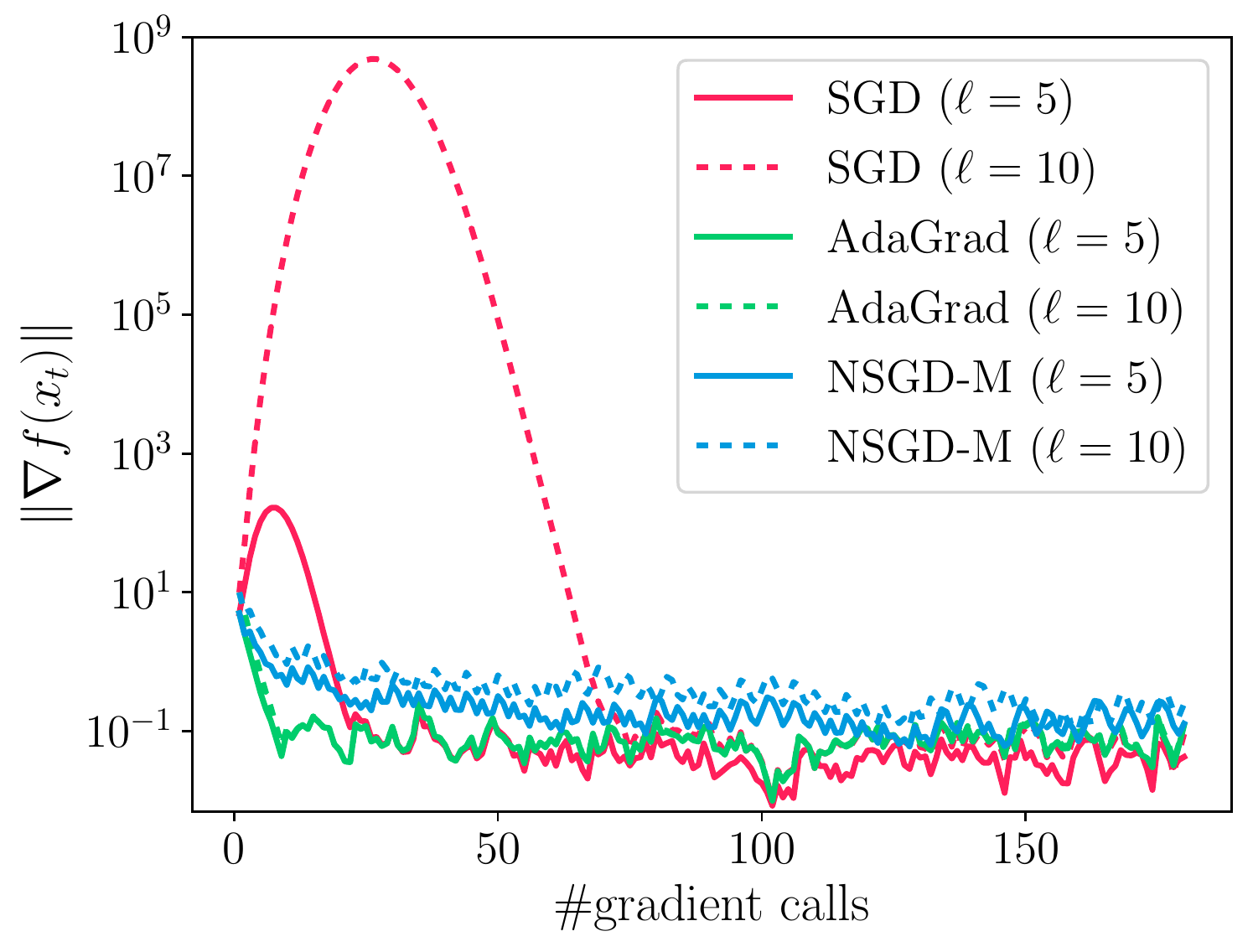}
      \label{subfig:quadratic}
    }
    \hspace{9mm}
    \subfigure[neural network]{
      \centering
      \includegraphics[width=0.45\textwidth]{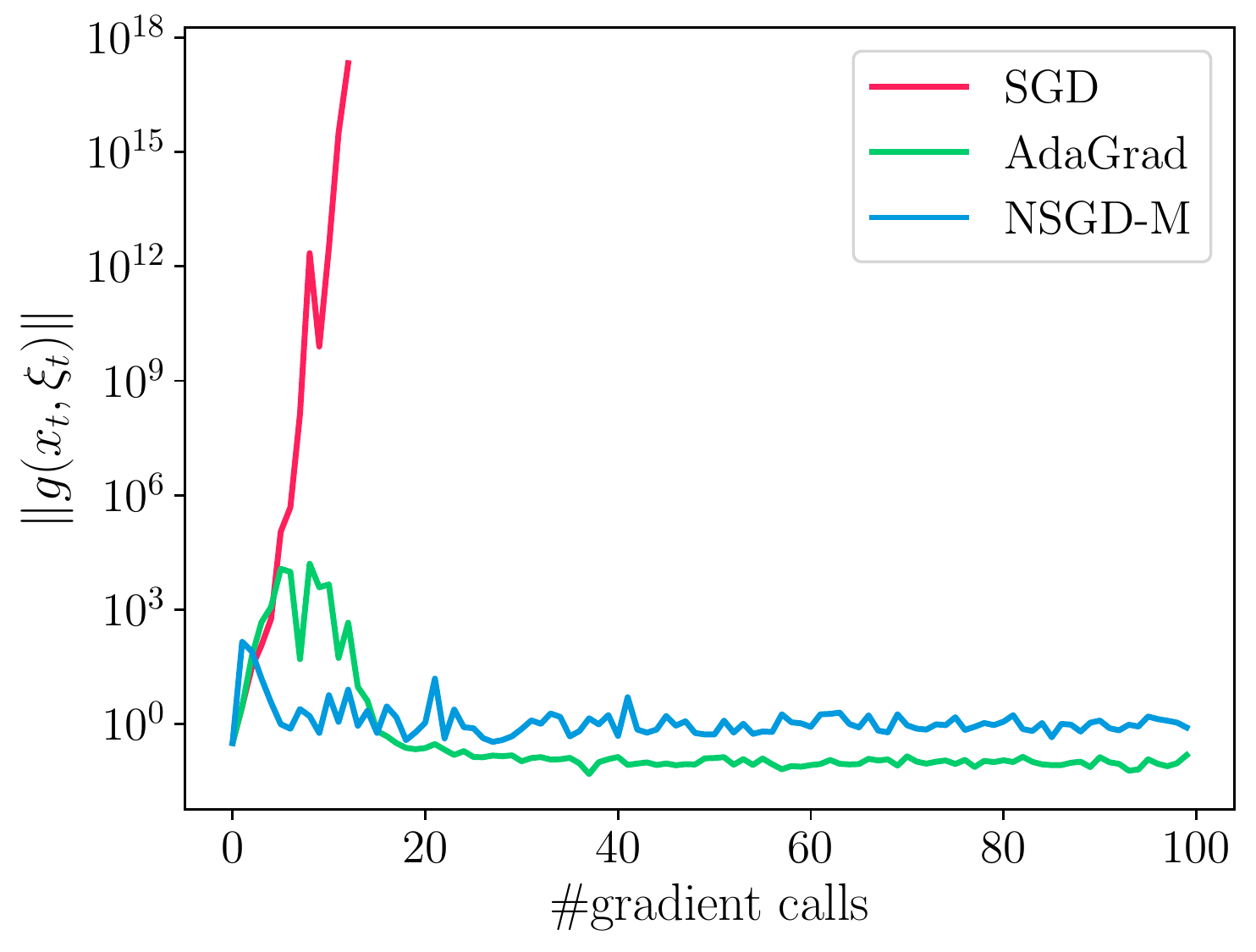}
      \label{subfig:mnist}
    }
    \caption{Comparison of SGD, AdaGrad, and NSGD-M on a quadratic function $f(x) = \ell x^2 / 2$ and a neural network. SGD employs a diminishing stepsize of $\eta/\sqrt{t}$, while the stepsizes for AdaGrad and NSGD-M are specified in Propositions \ref{thm:nsgdm} and \ref{prop:adagrad}, respectively. In the left figure, we set $\eta=1$ for all methods and test with two different values of $\ell$. In the right figure, we train a 3-layer neural network on the MNIST dataset \citep{lecun1998mnist} using cross-entropy loss and set $\eta=10$.
    }
    \label{fig:sgd_vs_adagrad}
\end{figure}

Consequently, this work is based on the premise of not assuming bounded gradients and hyper-parameters being independent of problem parameters. The main contributions are as follows:
\begin{itemize}[leftmargin=*]
    \item We show that untuned SGD with diminishing stepsizes $\eta_t = \eta/\sqrt{t}$ finds an $\epsilon$-stationary point of an $\ell$-smooth function within $\widetilde{\cO}((\ell^2 + \sigma^4 \eta^4 \ell^4 ) (4e)^{2\eta^2\ell^2}\epsilon^{-4})$ 
    iterations for any $\eta > 0$. Here $\sigma^2$ corresponds to the variance of the stochastic gradient. 
    Although it converges and has the optimal dependence on $\epsilon$, we show that the disastrous exponential term in ${\eta^2\ell^2}$ is unavoidable even when the algorithm has access to exact gradients. This explains its proneness to gradient explosion when the problem parameter is unknown. Previous analyses fail to capture this exponential term, as they assume $\eta$ is well-tuned to be $\Theta(1/\ell)$.
    \item AMSGrad, proposed to fix the nonconvergence of Adam, is not yet well-understood, with previous analyses depending on \textit{bounded stochastic gradients}. We show that AMSGrad (norm version) is free from exponential constants 
    in the deterministic setting without tuning, in stark contrast with SGD. Surprisingly, in the stochastic setting when the stochastic gradients are unbounded, we show that AMSGrad may converge at an arbitrarily slow polynomial rate. To the best of our knowledge, these are the first results of AMSGrad without assuming bounded gradients.
    \item To further illuminate the advantages of adaptive methods, we re-examine the results for Normalized Gradient Descent (NGD), Normalized SGD with momentum (NSGD-M) from \citep{cutkosky2020momentum} and AdaGrad-norm from \citep{yang2022nest}, considering stepsize independent of the problem parameters similar to untuned SGD. They all achieve near-optimal complexities while shredding off the exponential factor.  As a side result, we provide a strong non-convergence result of NSGD without momentum under any bounded stepsizes, which might be of independent interest.

\end{itemize}

Our findings contribute a fresh understanding of the performance gap between SGD and adaptive methods. Albeit with a near-optimal rate, untuned SGD is vulnerable to gradient explosion and slow convergence due to a large exponential constant in its complexity, which can be circumvented by several adaptive methods. To the best of our knowledge, this substantial difference is unformed in the previous literature, because the majority of analyses for SGD and adaptive methods turn to either well-tuned stepsize based on problem parameters or the assumption of bounded gradients. Part of our results are summarized in Table~\ref{table:summary_results1} and full results for a broader range of stepsizes can be found in Table~\ref{table:summary_results2} in the appendix.

\subsection{Related Work}

\paragraph{SGD in nonconvex optimization.} Stochastic approximation methods and SGD in particular have a long history of development \citep{robbins1951stochastic,Kiefer_Wolfowitz_1952,blum1954multidimensional,chung1954stochastic,nemirovski1983problem,polyak1992acceleration}. %
The last decade has witnessed fast progress
in the analysis of such methods for 
nonconvex optimization. 
When the objective is $\ell$-smooth and the gradient noise has bounded variance $\sigma^2$, \citet{ghadimi2013stochastic} and \citet{bottou2018optimization} prove that if $\eta_t = \eta/\sqrt{T}$, where $\eta = \eta(\ell, \sigma^2)$ and $T$ is the total iteration budget, then SGD finds an $\epsilon$-stationary point within $\cO(\ell\sigma^2\epsilon^{-4})$ iterations. Similar complexity (up to a logarithmic term) can also be achieved by decaying stepsizes $\eta/\sqrt{t}$ \citep{ghadimi2013stochastic, drori2020complexity, wang2021convergence}. This result was later shown to be optimal for first-order methods under these assumptions \citep{arjevani2022lower}. Several works consider various relaxations of the stochastic oracle model with bounded variance, for instance, biased oracle \citep{Ajalloeian_Stich_2021} or
expected smoothness \citep{khaled2020better}. However, these results also heavily rely on sufficiently small $\eta$, e.g., $\eta \leq 1/\ell$, and the convergence behavior in the large $\eta$ regime is rarely discussed. Remarkably, \citet{lei2019stochastic} characterize %
the convergence of SGD under individual smoothness and unbiased function values instead of bounded variance. 
They consider Robbins-Monro type stepsize schemes, which also includes $\eta/t^{\alpha}$ when $\alpha > 1/2$, and derive $\cO(\epsilon^{\frac{2}{\alpha - 1}})$ sample complexity %
including an exponential dependence on individual smoothness parameters.
We note that unlike \citep{lei2019stochastic}, we focus on the standard assumptions and derive better dependency in smoothness constant when $\alpha > 1/2$.
Importantly, we further 
justify that the exponential constants are unavoidable by providing a matching lower bound.

\paragraph{Adaptive methods.} We will restrict our focus to methods using gradients to adjust stepsize, rather than other strategies, such as backtracking line search \citep{armijo1966minimization}. Normalized Gradient Descent (NGD), which takes the stepsize to be normalized by the current gradient, was introduced by \citep{nesterov1984minimization} to minimize quasi-convex functions. \citet{hazan2015beyond} apply NGD and NSGD with minibatch to a wider class of locally-quasi-convex functions. They also show that minibatch is necessary for NSGD with constant learning rates. Later, \citet{cutkosky2020momentum} and \citet{zhao2021convergence} prove NSGD with momentum or minibatch, respectively, can find an $\epsilon$-stationary point in smooth nonconvex optimization with sample complexity $\cO(\epsilon^{-4})$. Different from NSGD, AdaGrad uses all past gradients and was initially introduced in the online convex optimization \citep{duchi2011adaptive, mcmahan2010adaptive}. In nonconvex optimization, AdaGrad and its scalar version, AdaGrad-norm \citep{streeter2010less}, achieve competitive convergence rate with SGD with decaying stepsizes \citep{ward2020adagrad, li2019convergence, kavis2022high, li2020high}. RMSProp \citep{hinton2012neural} and Adam \citep{kingma2015adam} use the decaying moving average of past gradients, but may suffer from divergence without hyper-parameter tuning \citep{reddi2018convergence}. Recently, it was shown that they can converge to a neighborhood, whose size shrinks to 0 by tuning hyper-parameters \citep{shi2021rmsprop, zhang2022adam}. Many variants of Adam are proposed,  including AMSGrad \citep{reddi2018convergence}, AdaBound \citep{luo2019adaptive}, Adan \citep{xie2022adan}, etc. However, most of these results on AdaGrad and Adam-type algorithms assume both Lipschitz and bounded gradients \citep{zhou2018convergence, chen2019convergence, defossez2020simple, ward2020adagrad, zou2019sufficient}. Very recently, \citet{Faw_2022_Power_Adapt} and \citet{yang2022nest} independently show that AdaGrad-norm converges without assuming bounded gradients and without the need for tuning. Moreover, its sample complexity is $\widetilde{\cO}(\epsilon^{-4})$. %

\paragraph{SGD v.s. adaptive methods.} Although we see above that SGD with polynomially decaying stepsize and adaptive stepsizes achieve similar complexities, adaptive methods are observed to have faster convergence in practice \citep{brown2020language, liu2020understanding} and are widely used to prevent large gradients \citep{pascanu2013difficulty, ginsburg2019stochastic}. %
There are many attempts in the literature trying to theoretically explain such differences.
Some works suggest that the advantage of adaptive algorithms is to achieve order-optimal rates without knowledge about the problem parameters, such as smoothness and variance of the noise \citep{ward2020adagrad, levy2021storm+, kavis2019unixgrad}. Another line of works investigates the ability of adaptive methods to escape saddle points faster \citep{levy2016power, murray2019revisiting, xie2022adaptive}. There are also works that  suggest the importance of taking the sign for each coordinate in Adam, which can be considered element-wise normalization \citep{balles2018dissecting,anonymous2023heavytailed}. In addition, \citet{zhang2020adaptive} study the influence of heavy-tail noise on the performance of adaptive methods, and \citet{zhou2020towards} and \citet{wang2021implicit} compare the generalization performance of non-adaptive and adaptive schemes. However, most of the previous works do not provide an explanation for the faster convergence of adaptive methods in the aspect of sample complexity. Notably, \citet{zhang2019gradient} and \citet{wang2022provable} explain the benefits of gradient clipping and Adam by analyzing them under a relaxed smoothness assumption, the setting where SGD with non-adaptive stepsizes may not even converge. Our paper is different from the recent work \citet{wang2022provable} in that we analyze SGD and several adaptive methods under the standard smoothness and noise assumptions, while they focus on the random-shuffling variant of Adam for finite-sum problems with individual relaxed smoothness. %

\section{Problem Setting}
\begin{algorithm}[ht] 
    \caption{Stochastic Gradient Descent (SGD)}
    \setstretch{1.23}
    \begin{algorithmic}[1]
      \STATE \textbf{Input:} initial point $x_0$
        \FOR{$t = 0,1,2,...$}
            \STATE sample  $\xi_t$ and set stepsize $\eta_t$
            \STATE  $x_{t+1} = x_t - \eta_t g(x_t; \xi_t)$
        \ENDFOR
    \end{algorithmic} \label{alg:sgd}
\end{algorithm}

Throughout this work, we focus on minimizing an $\ell$-smooth function $f: \mathbb{R}^d \rightarrow \mathbb{R}$.  We %
have access to a stochastic gradient oracle that returns $g(x; \xi)$ at any point $x$,
and we make the following standard assumptions in nonconvex optimization.

\begin{assumption}[smoothness] \label{assume:smoothness}
  Function $f(x)$ is $\ell$-smooth with $\ell > 0$, that is, for
  $\forall x_1, x_2 \in \bR^{d}$,
  \[
    \|\nabla f(x_1) - \nabla f(x_2) \| 
    \leq \ell \norm*{x_1 - x_2} .
  \]
\end{assumption}

\begin{assumption}[stochastic gradients]
\label{assume:stochastic_grad} 
The stochastic gradient $g(x; \xi) $ is unbiased and has a bounded variance, that is, for any $x \in \bR^{d}$, 
$$
\Ep[\xi]{g(x; \xi)} = \nabla f(x), \quad \Ep[\xi]{\norm*{g(x; \xi) - f(x)}^2} \leq \sigma^2.
$$
\end{assumption}

We present the general scheme of Stochastic Gradient Descent in Algorithm \ref{alg:sgd} with initial point $x_0$ and a stepsize sequence $\{\eta_t\}_{t=0}^\infty$. Some commonly used stepsizes include polynomially and geometrically decaying stepsize, constant stepsize, cosine stepsize, etc. When the stepsize depends on the instantaneous or past gradients, i.e., $\{ g(x; \xi_k)\}_{k\leq t}$, we call it adaptive stepsize, namely Normalized SGD \citep{hazan2015beyond}, AdaGrad \citep{duchi2011adaptive}, Adam \citep{kingma2015adam}, AMSGrad~\citep{reddi2018convergence}, etc.
We also assume $f(x_0) - \min_{x\in \mathbb{R}^d} f(x) \leq \Delta$.
Since the function class we consider is nonconvex, we aim  to find  an $\epsilon$-stationary point $x$ with $\mathbb{E}[\norm*{\nabla f(x)}] \leq \epsilon$.

\section{Convergence of Untuned SGD}
\label{sec:sgd}
In this section, we focus on SGD with the decaying stepsize:
\begin{equation*}
\eta_t = \frac{\eta}{\sqrt{t+1}},
\end{equation*}
where $\eta > 0$ is the initial stepsize. Most convergent analysis  requires  $\eta < 2/\ell$ \citep{ghadimi2013stochastic, bottou2018optimization} so that there is ``sufficient decrease" in function value after each update,  and if $\eta$ is carefully chosen, it can achieve the near-optimal complexity of $\widetilde{\cO}(\ell \epsilon^{-4}\sigma^2)$ \citep{arjevani2022lower}.  Nevertheless, as the smoothness parameter is usually unknown, providing guarantees with optimal $\eta$ or assuming $\eta$ to be problem-dependent does not give enough insights into practical training with SGD. Hence we are interested in its convergence behavior in both small and large initial stepsize regimes, i.e., $\eta \leq 1/\ell$ and $\eta >1/ \ell$.

\begin{theorem}
\label{thm:unbounded}
Under Assumptions \ref{assume:smoothness} and \ref{assume:stochastic_grad}, if we run SGD with stepsize $\eta_t = \frac{\eta}{\sqrt{t + 1}}$, where $\eta > 0$,  
\begin{equation*}
\begin{split}
\frac{1}{T}\sum_{t=0}^{T-1} \mathbb{E}  \norm*{\nabla f(x_t)}^2  \leq  
    \begin{dcases}
         \frac{2A}{\eta\sqrt{T}} , &\text{ when } \eta \leq 1/\ell, \\
         \frac{4\sqrt{2} \ell A(4e)^{\tau}}{\sqrt{\pi T}}, &\text{ when } \eta > 1/\ell, \\
    \end{dcases}
\end{split}
\end{equation*}
where  $\tau = \lceil \eta^2\ell^2 - 1\rceil$ and $A = \left(\Delta +  \frac{\ell\sigma^2\eta^2}{2}(1+\log T)\right)$.
\end{theorem}

This theorem  implies that when the initial stepsize $\eta > 1/\ell$, SGD still converges with a sample complexity of $\widetilde{\cO}((\ell^2 + \sigma^4 \eta^4 \ell^4 ) (4e)^{2\eta^2\ell^2}\epsilon^{-4})$ . Although the dependency in the target accuracy $\epsilon$ is near-optimal, it includes a disastrous exponential term in $\eta^2\ell^2$. This is due to  polynomially decaying stepsizes: in the first stage before $\tau = \lceil (\eta\ell)^2 - 1 \rceil$ iterations, the function value and gradients may keep increasing in expectation until reaching an exponential term in $\eta^2\ell^2$, which is in stark contrast with adaptive methods that we will see in  Section \ref{sec:adaptive}; in the second stage after $t \geq \tau$, the stepsize is small enough to decrease the function value in expectation at a rate of $1/\sqrt{T}$ up to a small term in $\sigma$.

If we pick an arbitrary $\eta = \Theta(1)$, untuned SGD may induce large gradients growing exponentially in $\ell$ in the first stage, which matches our observation in Figure~\ref{fig:sgd_vs_adagrad}. On the other hand, deriving the dependence in hyper-parameter $\eta$ is essential for assessing the effort required in its tuning: SGD with $\eta$ that is $c > 1$ times larger than the optimally tuned one can have an $\exp(\text{poly}(c))$ times larger gradient norm in the convergence guarantee.
To the best of our knowledge, there is  limited study for non-asymptotic analysis of untuned SGD under the same assumptions. \citet{moulines2011non} study untuned SGD under individual smoothness and convexity assumptions, i.e., $g(x; \xi)$ is Lipschitz continuous and $F(x; \xi)$ is convex almost surely. They show an $\cO(1/T^{1/3})$ rate, which is suboptimal in the convex case.
Later, \citet{fontaine2021convergence} provide $\cO(1/T^{1/2})$ convergence rate for untuned SGD in the convex setting yet without an explicit dependency in $\ell$ and $\eta$.

\begin{remark}
We focus on the stepsize of the order of $1/\sqrt{t}$, because it is known for SGD to achieve the best dependency in $\epsilon$ for nonconvex optimization \citep{drori2020complexity} and easier to compare with adaptive stepsizes. We also present the convergence results for more general polynomially decaying stepsizes, i.e., $\eta_t = \frac{\eta}{(t + 1)^{\alpha}}$ with $0 < \alpha < 1$, in Theorem \ref{thm:general_stepsize} of the appendix. There exists a trade-off between convergence speed $\cO(1/T^{\frac{1-\alpha}{2}})$ and the exponential term in $(\eta\ell)^{1/\alpha}$ for $\alpha \in [1/2, 1)$. Intuitively, larger $\alpha$ leads to a shorter time in adapting to $1/\ell$ stepsize but a slower convergence rate. We do not consider constant stepsize, i.e., $\alpha = 0$, because it is well known to diverge even in the deterministic setting if the stepsize is agnostic to the problem parameter  \citep{nesterov2003introductory, ahn2022understanding}. 
\end{remark}

The question arises as to whether the exponential term is necessary. In the following, we provide a lower bound for SGD under this choice of stepsize.

\begin{figure}[t]
    \centering
    \includegraphics[width=0.9\linewidth]{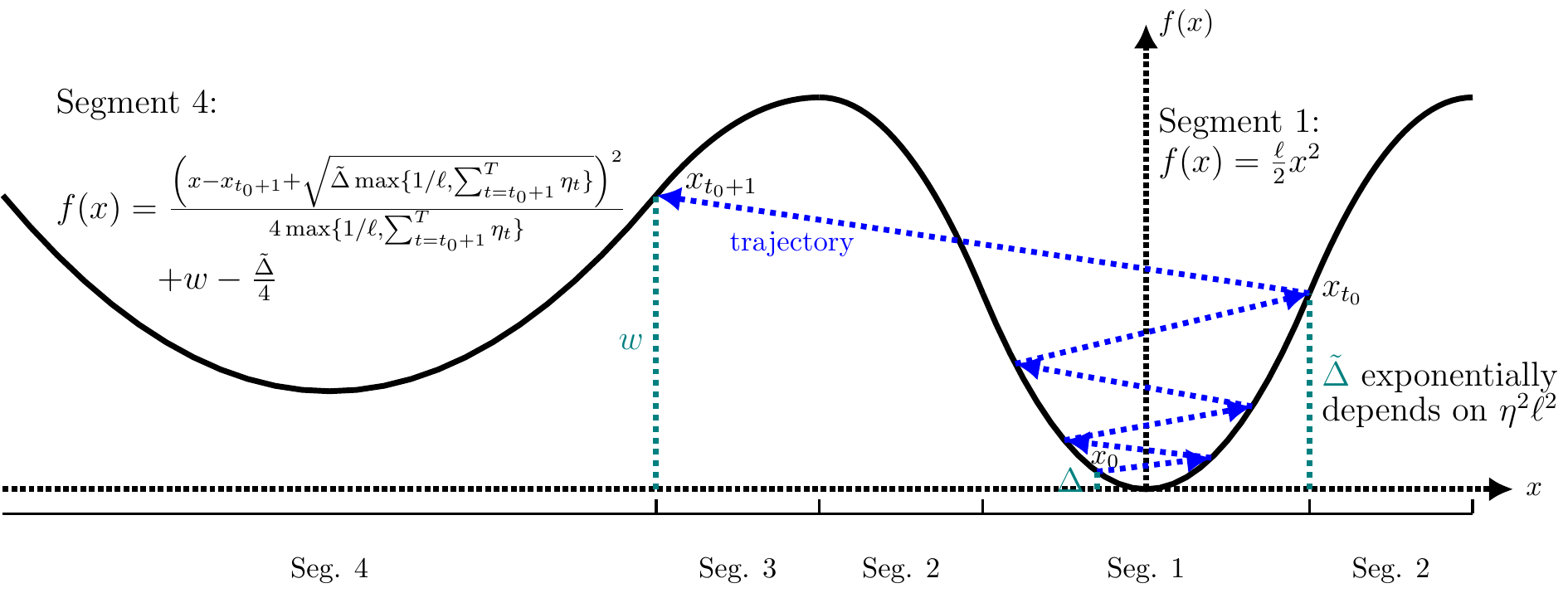}
    \caption{Demonstration of the constructed function for proving the lower bound.}
    \label{fig:lower_bound_demo}
\end{figure}

\begin{theorem} \label{thm:lower_bound}
Fixing $T \geq 1, \eta > 0, \ell>0$ and $\Delta > 0$ that $\eta\ell \geq 5$, there exists a $\ell$-smooth function $f: \mathbb{R} \rightarrow \mathbb{R}$ and an initial point $x_0$ with $f(x_0) - f^* \leq \Delta$ such that if we run Gradient Descent with stepsize $\eta_t = \frac{\eta}{\sqrt{t+1}}$, then for $t\leq t_0 = \left \lfloor{\eta^2\ell^2/16-1}\right \rfloor
$, 
\begin{align*}
     |\nabla f(x_{t})| \geq \sqrt{\frac{2\ell\Delta}{3\sqrt{t}}}\left(8 e \right)^{t/2}   \text{ and }
     |\nabla f(x_{t_0})| \geq \sqrt{\frac{8\Delta}{3\eta}}\left(8 e \right)^{\eta^2\ell^2/32 - 4};
\end{align*}
 if $T > t_0$, then for $t_0 < t \leq T$,
\begin{equation*}
     |\nabla f(x_t)| \geq \frac{1}{4}\sqrt{\Tilde{\Delta}}\min\bigg\{\ell^{1/2}, (2\eta)^{-1/2}T^{-1/4} \bigg\}, \text{ where } \Tilde{\Delta} \geq \frac{4}{3\eta\ell}(8e)^{\eta^2\ell^2/16-2}\Delta.
\end{equation*}
\end{theorem}

This theorem suggests that Gradient Descent with decaying stepsize $\eta/\sqrt{t+1}$ needs at least \\
$\Omega(\eta^{-4}\ell^{-2}(8e)^{\eta^2\ell^2/8} \epsilon^{-4})$ iterations to find an $\epsilon$-stationary point in the large initial stepsize regime. Therefore, it justifies that an exponential term in $\eta^2\ell^2$ multiplied by $1/\sqrt{T}$ is not avoidable even in the deterministic setting. Note that our result is limited to untuned (S)GD with the particular stepsize scheme.  It is worth pointing out that the existing lower bounds for first-order methods~\citep{arjevani2022lower} and SGD \citep{drori2020complexity} do not contain any exponential terms.

We illustrate our hard instance for Theorem \ref{thm:lower_bound} in Figure \ref{fig:lower_bound_demo}, which is one-dimensional. The algorithm starts from a valley of the function $f(x) = \ell x^2 /2$, i.e., Segment 1. Because of the large initial stepsize and steep slope, in the first $t_0$ iterations, Gradient Descent increases the function value as large as $\Tilde{\Delta} = \Omega\left((8e)^{\eta^2\ell^2/16}\Delta \right)$. Then the iterate $x_{t_0 + 1}$ jumps to the top of a very flat valley, i.e., Segment 4, so that Gradient Descent decreases the gradient as slowly as $\Omega(T^{-1/4})$. %

\textit{Why do not we assume gradients to be bounded?} The assumption on bounded gradients is not satisfied even for the simple function $f(x) = \ell x^2/2$. When training neural networks,  gradient explosion is often observed \citep{pascanu2013difficulty,schmidhuber2015deep}, which directly suggests that this assumption is not satisfied or only satisfied with a numerically large constant. In Proposition \ref{lemma:sgd_with_bound_g} in the appendix, we also provide a simple proof for the convergence  under the additional assumption of bounded gradient, i.e., $\|\nabla f(x)\|\leq G$ for all $x$,  attaining a sample complexity of $\widetilde{\cO}(\eta^2\ell^2G^4\sigma^2\epsilon^{-4})$ without any information about problem parameters. However, compared with Theorem \ref{thm:unbounded} and \ref{thm:lower_bound}, constant $G$ hides the exponential term. In Figure \ref{fig:sgd_vs_adagrad}, we observe that the gradient bound along the trajectory of non-adaptive stepsize  can be much larger than that of adaptive stepsize even if starting from the same initial point, so assuming bounded gradient will obscure the difference between them. 

\section{Power of Adaptive Methods}
\label{sec:adaptive}

In this section, we focus on the convergence behaviors of adaptive methods, which adjust their stepsizes based on the observed gradients. In particular, when arriving at a point with a large gradient, adaptive methods automatically decrease their stepsizes to counter the effect of possible gradient increase; to list a few, Normalized SGD \citep{hazan2015beyond}, AdaGrad \citep{duchi2011adaptive}, Adam \citep{kingma2015adam}. Since the analysis for adaptive methods is usually on a case-by-case basis, we will examine three examples -- Normalized SGD,  AMSGrad-norm, and AdaGrad-norm --  to 
establish a universal observation that
they avoid exponential dependency in $\ell$ without tuning.
Although many existing analyses rely on bounded gradients (and function values) or information on problem parameters, we will abandon such assumptions as noted in the previous section. We focus on the norm instead of the coordinate-wise version of adaptive methods, which means each coordinate adopts the same stepsize, because the norm version is usually  dimension-independent in the complexity, and is also widely used in both theory and practice \citep{zhang2018improved, ling2022vectoradam, li2019convergence, leevy2020survey, palfinger2022continuous, kavis2022adaptive}.

 \subsection{Family of Normalized SGD}

Normalized (Stochastic) Gradient Descent \citep{nesterov1984minimization, hazan2015beyond}, referred to as NGD and NSGD, is one of the simplest adaptive methods. It takes the stepsize in  Algorithm~\ref{alg:sgd} to be normalized by the norm of the current (stochastic) gradient:
\begin{equation*}
\eta_t = \frac{\gamma_t}{\|g(x_t; \xi_t)\|},
\end{equation*}
where $\{ \gamma_t\}_{t\geq 0}$ is a sequence of positive learning rate. \citet{cutkosky2020momentum} and \citet{zhao2021convergence} 
show that NSGD with $\gamma_t = \gamma/{\sqrt{T}}$ can find an $\cO(1/\sqrt{T} + \sigma)$-stationary point. 
In order to compare fairly with untuned SGD with decaying stepsize, we present a  modification with decaying $\gamma_t = \gamma/\sqrt{t +1 }$ in NSGD.

\begin{proposition} \label{thm:NSGD}
    Under Assumption~\ref{assume:smoothness} and \ref{assume:stochastic_grad}, if we run NSGD with $\gamma_t = \frac{\gamma}{\sqrt{t + 1}}$, then for any $\gamma > 0$, 
\begin{equation*}
\frac{1}{T} \sum_{t = 0}^{T-1}  \mathbb{E}\|\nabla f(x_t)\| \leq   3 \left(  \frac{\Delta}{\gamma} +   \ell \gamma \log(T) \right) T^{-\nicefrac{1}{2}} + 24 \sigma .
\end{equation*}
\end{proposition}

\paragraph{NGD.} In the deterministic setting, by Proposition \ref{thm:NSGD}, NGD converges to an $\epsilon$-stationary point with a complexity of $\widetilde{\cO}((\gamma^{-2}+\gamma^2\ell^2)\epsilon^{-2})$ for any $\gamma>0$, which importantly does not include any exponential term. Thus, even if the initial stepsize is not small enough, it does not result in a catastrophic gradient explosion.

\paragraph{NSGD.}
In the stochastic setting, Proposition \ref{thm:NSGD} implies that NSGD can find an $\epsilon$-stationary point only when the noise variance is small enough, i.e., $\sigma \leq \cO(\epsilon)$. This is not the consequence of a loose analysis. 
 \citet{hazan2015beyond} show that NSGD with constant $\gamma_t \equiv \gamma$ does not converge when the mini-batch size is smaller than $\Theta(\epsilon^{-1})$ for a non-smooth convex function. Here we provide a non-convergence result in the gradient norm with a smooth objective for all uniformly bounded stepsizes.
 The intuition behind this is illustrated in Figure \ref{fig:nsgd_diverge} in the appendix, where $\mathbb{E}_\xi \ g(x; \xi)/\|g(x; \xi)\|$ can easily vanish or be in the opposite direction of $\nabla f(x)$ under some noises.

\begin{theorem} \label{lemma:nonconvergence}
Fixing $\ell > 0$, $\sigma > 0$, $\epsilon > 0$, $\Delta > 0$ and stepsize sequence $\{\gamma_t\}_{t=0}^\infty$ with $\gamma_t \leq \gamma_{\max} $ that %
$\epsilon^2 < \min\{\sigma^2, 2\ell\Delta,  2\Delta(\sigma - \epsilon) / \gamma_{\max} \}$, there exists an $\ell$-smooth convex function $f$, initial point $x_0$ with $ f(x_0) - \min_x f(x) \leq \Delta$ and zero-mean noises with $\sigma^2$ variance such that the output from  NSGD satisfies $\mathbb{E}\|\nabla f(x_t)\| \geq \epsilon$ for all $t$.
\end{theorem}

This theorem implies that fixing function class ($\ell, \Delta, \sigma$) and any sequence $\{\gamma_t\}_t$ uniformly upper bounded by $\gamma_{\max}$, NSGD cannot converge to an arbitrarily small $\epsilon$. Specifically, the expected gradient norm will stay larger than $\min\{\sigma, \sqrt{2\ell\Delta}, \gamma_{\max}^{-1}(-\Delta+\sqrt{\Delta^2 + 2\Delta\gamma_{\max}\sigma})\}$. Most $\{\gamma_t\}_t$ used in practice is upper bounded, e.g., constant or decreasing sequences.
The condition $\epsilon^2 < 2\ell\Delta$ is necessary by noting that $\|\nabla f(x_0)\|^2 \leq 2\ell [f(x_0) - \min_xf(x)] \leq 2\ell\Delta$. Considering $\gamma_t = 1/\sqrt{t+1} $, when $\Delta \geq \sigma$ and $\sqrt{2\ell\Delta} \geq \sigma$, it matches with Proposition \ref{thm:NSGD} where NSGD can only converge to a $\Theta(\sigma)$-stationary point.  Since Sign-SGD and NSGD coincide in one-dimensional objectives, our non-convergent example also applies to Sign-SGD. It sheds light on why increasing batch size improves Normalized and Sign-SGD \citep{zhao2021convergence,anonymous2023heavytailed}. However, they are generally different in higher dimensions,
 and \citet{karimireddy2019error} show that sign-SGD may not converge even with full-batch. 

\paragraph{NSGD with momentum.} 
While NSGD may not always converge, \citet{cutkosky2020momentum} introduced NSGD with momentum (NSGD-M) presented in Algorithm \ref{alg:nsgdm} with constant $\gamma_t \equiv \gamma$. We provide the following modification with diminishing $\gamma_t$ that eliminates the need to specify the total number of runs beforehand.

\begin{floatalgo}
\begin{minipage}[H]{0.46\textwidth}
\begin{algorithm}[H] 
\setstretch{1.47}
    \caption{NSGD-M}
    \begin{algorithmic}[1]
      \STATE \textbf{Input:} initial point $x_0$,  stepsize sequence $\{\gamma_t\}$, momentum sequence $\{\alpha_t\}$, and initial momentum $g_0 $.
        \FOR{$t = 0,1,2,...$}
            \STATE  $x_{t+1} = x_t -  \frac{\gamma_t}{\|g_t\|} g_t$
            \STATE sample  $\xi_{t+1}$
            \STATE  $g_{t+1} = (1-\alpha_{t}) g_{t} + \alpha_{t} g(x_{t+1}; \xi_{t+1})$
        \ENDFOR
    \end{algorithmic} \label{alg:nsgdm}
\end{algorithm}
\end{minipage}
\hfill
\begin{minipage}[H]{0.52\textwidth}
\begin{algorithm}[H] 
    \caption{AMSGrad-norm}
    \begin{algorithmic}[1]
      \STATE \textbf{Input:} initial point $x_0$, momentum parameters $0 \leq \beta_1 < 1$ and $0 \leq \beta_2 \leq 1$, stepsize sequence $\{\gamma_t\}$
      and initial momentum $m_0$ and $v_0 > 0$.
      \STATE $\hat v_0 = v_0$
        \FOR{$t = 0,1,2,...$}
            \STATE sample  $\xi_{t}$
            \STATE $m_{t+1} = \beta_1 m_t + (1 - \beta_1) g(x_t; \xi_t)$
            \STATE $v^2_{t+1} = \beta_2 v^2_t + (1 - \beta_2) \|g(x_t; \xi_t)\|^2$
            \STATE $\hat v^2_{t+1} = \max\{\hat v^2_t, v^2_{t+1}\}$
            \STATE  $x_{t+1} = x_t - \frac{\gamma_t}{\sqrt{\hat v^2_{t+1}}} m_{t+1}$
        \ENDFOR
    \end{algorithmic} \label{alg:amsgrad_full}
\end{algorithm}
\end{minipage}
\end{floatalgo}

\begin{proposition}\label{thm:nsgdm}
Under Assumptions \ref{assume:smoothness} and \ref{assume:stochastic_grad}, if we run NSGD-M with $\alpha_t = \frac{\sqrt{2}}{\sqrt{t+2}}$ and $\gamma_t = \frac{\gamma}{(t + 1)^{3/4}}$, then for any $\gamma > 0$, %
\begin{equation*}
\frac{1}{T} \sum_{t = 0}^{T-1}  \mathbb{E}\|\nabla f(x_t)\| \leq \frac{ C \left(  \frac{\Delta}{\gamma} +  (\sigma  + \ell \gamma) \log(T) \right) }{T^{\frac{1}{4}}} ,
\end{equation*}
where $C > 0$ is a numerical constant. 
\end{proposition}

It implies that NSGD-M attains a complexity of $\widetilde{\cO}((\gamma^{-4}+\gamma^4\ell^4)\epsilon^{-4})$ for any $\gamma > 0$. Compared with Theorem \ref{thm:unbounded} and \ref{thm:lower_bound}, NSGD-M not only achieves near-optimal dependency in the target accuracy $\epsilon$, but also shreds the exponential term when the hyper-parameter is agnostic to smoothness constant.

\subsection{AMSGrad-norm}

AMSGrad was introduced by \citet{reddi2018convergence} to fix the possible non-convergence issue of Adam.  Notably, current analyses of AMSGrad in the stochastic setting show a convergence rate of $\widetilde{\cO}(1/T^{1/4})$, but they rely on the assumption of \textit{bounded stochastic gradients} \citep{chen2019convergence, zhou2018convergence}, which is much stronger than assumptions used for SGD analysis. Here, we examine the simpler norm version of AMSGrad, presented in Algorithm \ref{alg:amsgrad_full}. We prove that without assuming bounded stochastic gradients, AMSGrad-norm with default $\gamma_t = \gamma/\sqrt{t + 1}$ 
may converge at an arbitrarily slow polynomial rate. In fact, this holds even if the true gradients are bounded. We believe this result is of independent interest.

\begin{theorem}
\label{thm:amsgrad_stoc}
    For any $\ell > 0$, $\Delta > 0$, $\sigma > 0$ and $T > 1$, there exists a
    $\ell$-smooth function $f: \mathbb{R}^2 \rightarrow \mathbb{R}^2$, $x_0$ with $f(x_0) - \inf_x f(x) \leq \Delta$ and
    noise distribution $P$ with variance upper bounded by $\sigma^2$, such that if we
    run AMSGrad-norm  with $0 \leq \beta_1 \leq 1$, $0 \leq \beta_2 < 1$ and $\gamma_t = \frac{\gamma}{\sqrt{t+1}}$, we have
    with probability $\frac{1}{2}$, it holds that
\begin{equation*}
        \min_{t \in \{0, 1, ..., T-1\}} \|\nabla f(x_t) \| \geq \sqrt{\frac{\Delta}{16 \max\left\{1/\ell,
        \frac{ \gamma \sqrt{2\Gamma\left(1-\frac{\zeta}{2}\right)} }{\sigma \left(e\left(\frac{1}{\zeta} - 1 \right)\right)^{\frac{\zeta}{2}} (1-\zeta) \sqrt{1-\beta_2}} \left(T^{1-\zeta} - \zeta \right) \right\}}}
\end{equation*}
    for any $\frac{1}{2} < \zeta < 1$, where $\Gamma(\cdot)$ denotes the Gamma function. %
\end{theorem}

The intuition behind this theroem is that since AMSGrad utilizes the maximum norm of past stochastic gradients with momentum in the denominator of stepsizes, some noise distributions enable this maximum norm to increase polynomially, making the stepsizes too small. However, we can still explore its benefit in the deterministic setting. Whether it converges without assuming bounded gradients, to the best of our knowledge, is unknown. Here, for simplicity, we consider AMSGrad-norm without momentum, i.e., $\beta_1 = \beta_2 = 0$.

\begin{theorem} \label{theorem:ams_grad}
    Under Assumption \ref{assume:smoothness},
    if we run  AMSGrad-norm with $\gamma_t=\frac{\gamma}{\sqrt{t+1}}$, $v_0 > 0$ and $\beta_1 = \beta_2 = 0$ in the deterministic setting, then
    for any $\gamma > 0$ and $0 < \alpha < 1$, 
    \begin{equation*}
            \frac{1}{T} \sum_{t=0}^{T-1} \|\nabla f(x_t)\| 
         \leq \begin{dcases} \frac{\sqrt{2 \Delta \max\{v_0, \sqrt{2\ell\Delta} \}} }{\sqrt{\gamma} T^{\frac{1}{4}}} , &\text{when} \ \  v_0 < \gamma \ell, \\
         \frac{ \gamma^2 \ell^2 }{v_0^2 \sqrt{T}} + \frac{\sqrt{2(M + \Delta)   \max\{\gamma\ell, \sqrt{2\ell(M+\Delta)} \} }}{\sqrt{\gamma} T^{\frac{1}{4}}}, &\text{when} \ \  v_0 \geq \gamma \ell,
         \end{dcases}
    \end{equation*}
    where $M = \ell\gamma^2 \left( 1+ \log \left( \frac{\ell\gamma}{v_0}\right)\right)$.
\end{theorem}

The theorem implies that AMSGrad-norm achieves a complexity of $\widetilde{\cO}((\ell^4\gamma^4+\ell^2+\ell^3\gamma^2 + \ell\gamma^{-2})\epsilon^{-4})$ with the default $\gamma_t = \Theta(t^{-1/2})$ \citep{reddi2018convergence, chen2019convergence, guo2021novel}. Compared with untuned Gradient Descent, it gets rid of the exponential dependency. In the proof, we show that before the first iteration $\tau$ when stepsize $\eta_t$ reduces to $1/\ell$, the accumulated gradient norms $\sum_{t=0}^{\tau-1}\left\|\nabla f\left(x_t\right)\right\|^2$ are upper bounded polynomially, which is in striking contrast with SGD in Theorem \ref{thm:lower_bound}. 
We further provide theoretical guarantees for more general schemes
$\frac{\gamma}{(t+1)^{\alpha}}$ with $0 < \alpha < 1$ in Theorem \ref{thm:amsgrad_general_ss} in the appendix. We also derive matching lower bounds  in Theorem~\ref{thm:amsgrad_deter_lower_bound} for any $0 < \alpha < 1$, and justify that AMSGrad may fail to converge with constant $\gamma_t \equiv \gamma$ (i.e., $\alpha = 0$) if the problem parameter is unknown.

\subsection{AdaGrad-norm}

AdaGrad chooses its stepsize to be inversely proportional to the element-wise accumulated past gradients \citep{duchi2011adaptive, mcmahan2010adaptive}. Its norm-version, AdaGrad-norm \citep{streeter2010less, ward2020adagrad}, picks stepsize in (\ref{alg:sgd}) to be 
\begin{equation*}
\eta_t = \frac{\eta}{\sqrt{ v_0^2 + \sum_{k = 0}^t \|g(x_k;\xi_k)\|^2}},
\end{equation*}
where $v_0 > 0$. Very recently, AdaGrad is proven to converge in nonconvex optimization without the assumption on bounded gradients or tuning $\eta$ \citep{Faw_2022_Power_Adapt, yang2022nest}. 
Although the result in \citep{yang2022nest} are presented for minimax optimization problems,  a similar result follows immediately for minimization problems. We present the following result for the completeness of the paper and to further illustrate the benefits of adaptive methods over SGD.

\begin{proposition}
\label{prop:adagrad}
  Under Assumptions \ref{assume:smoothness} and \ref{assume:stochastic_grad}, if we run AdaGrad-norm, then for any $\eta >0$ and $v_0 > 0$,
\begin{equation*}
    \frac{1}{T} \sum_{t = 0}^{T-1}  \mathbb{E}\|\nabla f(x_t)\| \leq \frac{2A}{\sqrt{T}} + \frac{\sqrt{v_0 A}}{\sqrt{T}} + \frac{2\sqrt{A \sigma}}{T^{\frac{1}{4}}},
\end{equation*}
where $ A = \widetilde{\cO}\left( \frac{\Delta}{\eta} + \sigma +  \ell\eta \right)$. 
\end{proposition}

The result implies a complexity of $\widetilde{\cO}\left((\eta^{-2} + \sigma^2 + \eta^2\ell^2)\left( \epsilon^{-2} + \sigma^2\epsilon^{-4}\right)\right)$. 
Notably, if we choose $\eta$ to be $1/\sqrt{\ell}$, it  achieves the optimal complexity in both $\ell$ and $\epsilon$ up to logarithmic terms like well-tuned SGD \citep{arjevani2022lower}. Even if $\eta$ is independent to $\ell$, AdaGrad-norm does not suffer from the exponential term present in untuned SGD. 
One of the intuitions behind this in the deterministic setting, drawn parallel to AMSGrad-norm, is that the accumulated squared gradient norm before the stepsize drops to $1/\ell$  will be upper bounded by a polynomial term (see Theorem 3.2 in \citep{li2022adaptive}). 
Another benefit of AdaGrad over other methods is to achieve optimal convergence rates simultaneously in deterministic and stochastic settings with the same hyper-parameters. This is sometimes described as ``noise adaptivity", which is beyond the scope of this paper.

\section{Conclusion and Future Directions}
In this work, we study stochastic gradient methods for minimizing smooth functions in the parameter-agnostic regime. Firstly, we show SGD with polynomially decaying stepsize $1/\sqrt{t}$ is able to converge with the order-optimal rate, with and without bounded gradients (Proposition~\ref{lemma:sgd_with_bound_g} and Theorem~\ref{thm:unbounded}). Its limitation lies in an unavoidable exponential term in $\ell^2$ when we do not assume bounded gradients (Theorem \ref{thm:lower_bound}). We demonstrate that several existing adaptive methods do not suffer from the exponential dependency, such as NGD, AdaGrad, AMSGrad-norm in the deterministic setting (Proposition \ref{thm:NSGD} and Theorem~\ref{theorem:ams_grad}), and NSGD-M, AdaGrad in the stochastic setting (Proposition \ref{thm:nsgdm} and Proposition \ref{prop:adagrad}). However, it does not mean adaptive methods are always better than SGD. We provide a non-convergence result for NSGD (Theorem \ref{lemma:nonconvergence}) and a slow convergence result for AMSGrad-norm (Thoerem \ref{thm:amsgrad_stoc}) in the stochastic case. A detailed exposition of these results can be found in Appendix~\ref{sec:summary}. We believe our results shed light on explaining commonly observed large gradients during training and provide a better theoretical understanding of the convergence behaviors of adaptive methods in the regime with unbounded stochastic gradients. 

The results can be extended in several directions. Firstly, it is interesting to understand whether similar benefits of adaptive methods persist for the high probability convergence guarantees and extend to other adaptive optimizers. 
Secondly, we emphasize the significance of eliminating the assumption of bounded gradients for more adaptive algorithms. 
Such an assumption can hide the dependence on $\ell$ and obscure the advantage over SGD.
Thirdly, based on our negative results concerning AMSGrad-norm, further exploration of the convergence properties of AMSGrad and its variants becomes interesting. This exploration could involve scenarios where true function gradients are unbounded, but additional assumptions can be made regarding the noise distribution.
Lastly, understanding the impact of adaptive algorithms on the optimization of possibly non-smooth nonconvex objectives, which frequently arise in the training of modern machine learning models, is another intriguing avenue for future research.

\section*{Acknowledgement}
The work is supported by ETH research grant and Swiss National Science Foundation (SNSF) Project Funding No. 200021-207343; Ilyas Fatkhullin is partly funded by ETH AI Center.

\bibliography{ref}

\begin{thebibliography}{78}
\providecommand{\natexlab}[1]{#1}
\providecommand{\url}[1]{\texttt{#1}}
\expandafter\ifx\csname urlstyle\endcsname\relax
  \providecommand{\doi}[1]{doi: #1}\else
  \providecommand{\doi}{doi: \begingroup \urlstyle{rm}\Url}\fi

\bibitem[Ahn et~al.(2022)Ahn, Zhang, and Sra]{ahn2022understanding}
Kwangjun Ahn, Jingzhao Zhang, and Suvrit Sra.
\newblock Understanding the unstable convergence of gradient descent.
\newblock \emph{arXiv preprint arXiv:2204.01050}, 2022.

\bibitem[Ajalloeian and Stich(2021)]{Ajalloeian_Stich_2021}
Ahmad Ajalloeian and Sebastian~U. Stich.
\newblock On the convergence of sgd with biased gradients.
\newblock \emph{arXiv preprint arXiv:2008.00051}, 2021.

\bibitem[Arjevani et~al.(2022)Arjevani, Carmon, Duchi, Foster, Srebro, and
  Woodworth]{arjevani2022lower}
Yossi Arjevani, Yair Carmon, John~C Duchi, Dylan~J Foster, Nathan Srebro, and
  Blake Woodworth.
\newblock Lower bounds for non-convex stochastic optimization.
\newblock \emph{Mathematical Programming}, pages 1--50, 2022.

\bibitem[Armijo(1966)]{armijo1966minimization}
Larry Armijo.
\newblock Minimization of functions having lipschitz continuous first partial
  derivatives.
\newblock \emph{Pacific Journal of mathematics}, 16\penalty0 (1):\penalty0
  1--3, 1966.

\bibitem[Balles and Hennig(2018)]{balles2018dissecting}
Lukas Balles and Philipp Hennig.
\newblock Dissecting adam: The sign, magnitude and variance of stochastic
  gradients.
\newblock In \emph{International Conference on Machine Learning}, pages
  404--413. PMLR, 2018.

\bibitem[Bengio et~al.(1994)Bengio, Simard, and Frasconi]{Bengio94_ltdep}
Y.~Bengio, P.~Simard, and P.~Frasconi.
\newblock Learning long-term dependencies with gradient descent is difficult.
\newblock \emph{IEEE Transactions on Neural Networks}, 5\penalty0 (2):\penalty0
  157--166, 1994.

\bibitem[Bengio(2009)]{Bengio09ftml}
Yoshua Bengio.
\newblock Learning deep architectures for {AI}.
\newblock \emph{Foundations and Trends in Machine Learning}, 2\penalty0
  (1):\penalty0 1--127, 2009.

\bibitem[Blum(1954)]{blum1954multidimensional}
Julius~R Blum.
\newblock Multidimensional stochastic approximation methods.
\newblock \emph{The Annals of Mathematical Statistics}, pages 737--744, 1954.

\bibitem[Bottou et~al.(2018)Bottou, Curtis, and
  Nocedal]{bottou2018optimization}
L{\'e}on Bottou, Frank~E Curtis, and Jorge Nocedal.
\newblock Optimization methods for large-scale machine learning.
\newblock \emph{Siam Review}, 60\penalty0 (2):\penalty0 223--311, 2018.

\bibitem[Brown et~al.(2020)Brown, Mann, Ryder, Subbiah, Kaplan, Dhariwal,
  Neelakantan, Shyam, Sastry, Askell, et~al.]{brown2020language}
Tom Brown, Benjamin Mann, Nick Ryder, Melanie Subbiah, Jared~D Kaplan, Prafulla
  Dhariwal, Arvind Neelakantan, Pranav Shyam, Girish Sastry, Amanda Askell,
  et~al.
\newblock Language models are few-shot learners.
\newblock \emph{Advances in neural information processing systems},
  33:\penalty0 1877--1901, 2020.

\bibitem[Carmon et~al.(2020)Carmon, Duchi, Hinder, and
  Sidford]{carmon2020lower}
Yair Carmon, John~C Duchi, Oliver Hinder, and Aaron Sidford.
\newblock Lower bounds for finding stationary points i.
\newblock \emph{Mathematical Programming}, 184\penalty0 (1-2):\penalty0
  71--120, 2020.

\bibitem[Chen et~al.(2019)Chen, Liu, Sun, and Hong]{chen2019convergence}
Xiangyi Chen, Sijia Liu, Ruoyu Sun, and Mingyi Hong.
\newblock On the convergence of a class of adam-type algorithms for non-convex
  optimization.
\newblock In \emph{International Conference on Learning Representations}.
  International Conference on Learning Representations, ICLR, 2019.

\bibitem[Chung(1954)]{chung1954stochastic}
Kai~Lai Chung.
\newblock On a stochastic approximation method.
\newblock \emph{The Annals of Mathematical Statistics}, pages 463--483, 1954.

\bibitem[Cutkosky and Mehta(2020)]{cutkosky2020momentum}
Ashok Cutkosky and Harsh Mehta.
\newblock Momentum improves normalized sgd.
\newblock In \emph{International conference on machine learning}, pages
  2260--2268. PMLR, 2020.

\bibitem[De~Gusmao et~al.(2011)De~Gusmao, Ortega, and
  Cordeiro]{de2011generalized}
Felipe~RS De~Gusmao, Edwin~MM Ortega, and Gauss~M Cordeiro.
\newblock The generalized inverse weibull distribution.
\newblock \emph{Statistical Papers}, 52:\penalty0 591--619, 2011.

\bibitem[D{\'e}fossez et~al.(2020)D{\'e}fossez, Bottou, Bach, and
  Usunier]{defossez2020simple}
Alexandre D{\'e}fossez, L{\'e}on Bottou, Francis Bach, and Nicolas Usunier.
\newblock A simple convergence proof of adam and adagrad.
\newblock \emph{arXiv preprint arXiv:2003.02395}, 2020.

\bibitem[Drori and Shamir(2020)]{drori2020complexity}
Yoel Drori and Ohad Shamir.
\newblock The complexity of finding stationary points with stochastic gradient
  descent.
\newblock In \emph{International Conference on Machine Learning}, pages
  2658--2667. PMLR, 2020.

\bibitem[Duchi et~al.(2011)Duchi, Hazan, and Singer]{duchi2011adaptive}
John Duchi, Elad Hazan, and Yoram Singer.
\newblock Adaptive subgradient methods for online learning and stochastic
  optimization.
\newblock \emph{Journal of machine learning research}, 12\penalty0 (7), 2011.

\bibitem[Fatkhullin et~al.(2023)Fatkhullin, Barakat, Kireeva, and
  He]{Fatkhullin_SPG_FND_2023}
Ilyas Fatkhullin, Anas Barakat, Anastasia Kireeva, and Niao He.
\newblock Stochastic policy gradient methods: Improved sample complexity for
  fisher-non-degenerate policies.
\newblock \emph{arXiv preprint arXiv:2302.01734}, 2023.

\bibitem[Faw et~al.(2022)Faw, Tziotis, Caramanis, Mokhtari, Shakkottai, and
  Ward]{Faw_2022_Power_Adapt}
Matthew Faw, Isidoros Tziotis, Constantine Caramanis, Aryan Mokhtari, Sanjay
  Shakkottai, and Rachel Ward.
\newblock The power of adaptivity in sgd: Self-tuning step sizes with unbounded
  gradients and affine variance.
\newblock In \emph{Proceedings of Thirty Fifth Conference on Learning Theory},
  volume 178 of \emph{Proceedings of Machine Learning Research}, pages
  313--355. PMLR, Jul 2022.

\bibitem[Fontaine et~al.(2021)Fontaine, De~Bortoli, and
  Durmus]{fontaine2021convergence}
Xavier Fontaine, Valentin De~Bortoli, and Alain Durmus.
\newblock Convergence rates and approximation results for sgd and its
  continuous-time counterpart.
\newblock In \emph{Conference on Learning Theory}, pages 1965--2058. PMLR,
  2021.

\bibitem[Ghadimi and Lan(2013)]{ghadimi2013stochastic}
Saeed Ghadimi and Guanghui Lan.
\newblock Stochastic first-and zeroth-order methods for nonconvex stochastic
  programming.
\newblock \emph{SIAM Journal on Optimization}, 23\penalty0 (4):\penalty0
  2341--2368, 2013.

\bibitem[Ginsburg et~al.(2019)Ginsburg, Castonguay, Hrinchuk, Kuchaiev,
  Lavrukhin, Leary, Li, Nguyen, Zhang, and Cohen]{ginsburg2019stochastic}
Boris Ginsburg, Patrice Castonguay, Oleksii Hrinchuk, Oleksii Kuchaiev, Vitaly
  Lavrukhin, Ryan Leary, Jason Li, Huyen Nguyen, Yang Zhang, and Jonathan~M
  Cohen.
\newblock Stochastic gradient methods with layer-wise adaptive moments for
  training of deep networks.
\newblock \emph{arXiv preprint arXiv:1905.11286}, 2019.

\bibitem[Goodfellow et~al.(2016)Goodfellow, Bengio, and
  Courville]{Goodfellow_2016}
Ian Goodfellow, Yoshua Bengio, and Aaron Courville.
\newblock \emph{{D}eep {L}earning}.
\newblock {The MIT Press}, 2016.

\bibitem[Guo et~al.(2021)Guo, Xu, Yin, Jin, and Yang]{guo2021novel}
Zhishuai Guo, Yi~Xu, Wotao Yin, Rong Jin, and Tianbao Yang.
\newblock A novel convergence analysis for algorithms of the adam family.
\newblock \emph{arXiv preprint arXiv:2112.03459}, 2021.

\bibitem[Hazan et~al.(2015)Hazan, Levy, and Shalev-Shwartz]{hazan2015beyond}
Elad Hazan, Kfir Levy, and Shai Shalev-Shwartz.
\newblock Beyond convexity: Stochastic quasi-convex optimization.
\newblock \emph{Advances in neural information processing systems}, 28, 2015.

\bibitem[Hinton et~al.(2012)Hinton, Srivastava, and Swersky]{hinton2012neural}
Geoffrey Hinton, Nitish Srivastava, and Kevin Swersky.
\newblock Neural networks for machine learning lecture 6a overview of
  mini-batch gradient descent.
\newblock \emph{Cited on}, 14\penalty0 (8):\penalty0 2, 2012.

\bibitem[Karimireddy et~al.(2019)Karimireddy, Rebjock, Stich, and
  Jaggi]{karimireddy2019error}
Sai~Praneeth Karimireddy, Quentin Rebjock, Sebastian Stich, and Martin Jaggi.
\newblock Error feedback fixes signsgd and other gradient compression schemes.
\newblock In \emph{International Conference on Machine Learning}, pages
  3252--3261. PMLR, 2019.

\bibitem[Kavis et~al.(2019)Kavis, Levy, Bach, and Cevher]{kavis2019unixgrad}
Ali Kavis, Kfir~Y Levy, Francis Bach, and Volkan Cevher.
\newblock Unixgrad: A universal, adaptive algorithm with optimal guarantees for
  constrained optimization.
\newblock \emph{NeurIPS}, 32, 2019.

\bibitem[Kavis et~al.(2022{\natexlab{a}})Kavis, Levy, and
  Cevher]{kavis2022high}
Ali Kavis, Kfir Levy, and Volkan Cevher.
\newblock High probability bounds for a class of nonconvex algorithms with
  adagrad stepsize.
\newblock In \emph{ICLR}, 2022{\natexlab{a}}.

\bibitem[Kavis et~al.(2022{\natexlab{b}})Kavis, Skoulakis, Antonakopoulos,
  Dadi, and Cevher]{kavis2022adaptive}
Ali Kavis, Stratis Skoulakis, Kimon Antonakopoulos, Leello~Tadesse Dadi, and
  Volkan Cevher.
\newblock Adaptive stochastic variance reduction for non-convex finite-sum
  minimization.
\newblock \emph{arXiv preprint arXiv:2211.01851}, 2022{\natexlab{b}}.

\bibitem[Khaled and Richt{\'a}rik(2020)]{khaled2020better}
Ahmed Khaled and Peter Richt{\'a}rik.
\newblock Better theory for sgd in the nonconvex world.
\newblock \emph{arXiv preprint arXiv:2002.03329}, 2020.

\bibitem[Kiefer and Wolfowitz(1952)]{Kiefer_Wolfowitz_1952}
J.~Kiefer and J.~Wolfowitz.
\newblock {Stochastic Estimation of the Maximum of a Regression Function}.
\newblock \emph{The Annals of Mathematical Statistics}, 23\penalty0
  (3):\penalty0 462 -- 466, 1952.

\bibitem[Kingma and Ba(2015)]{kingma2015adam}
Diederik~P. Kingma and Jimmy Ba.
\newblock Adam: {{A}} method for stochastic optimization.
\newblock In \emph{{{ICLR}}}, 2015.

\bibitem[Kunstner et~al.(2023)Kunstner, Chen, Lavington, and
  Schmidt]{anonymous2023heavytailed}
Frederik Kunstner, Jacques Chen, Jonathan~Wilder Lavington, and Mark Schmidt.
\newblock Heavy-tailed noise does not explain the gap between {SGD} and {Adam},
  but sign descent might.
\newblock In \emph{International Conference on Learning Representations}, 2023.

\bibitem[LeCun(1998)]{lecun1998mnist}
Yann LeCun.
\newblock The mnist database of handwritten digits.
\newblock \emph{http://yann. lecun. com/exdb/mnist/}, 1998.

\bibitem[Leevy and Khoshgoftaar(2020)]{leevy2020survey}
Joffrey~L Leevy and Taghi~M Khoshgoftaar.
\newblock A survey and analysis of intrusion detection models based on
  cse-cic-ids2018 big data.
\newblock \emph{Journal of Big Data}, 7\penalty0 (1):\penalty0 1--19, 2020.

\bibitem[Lei et~al.(2019)Lei, Hu, Li, and Tang]{lei2019stochastic}
Yunwen Lei, Ting Hu, Guiying Li, and Ke~Tang.
\newblock Stochastic gradient descent for nonconvex learning without bounded
  gradient assumptions.
\newblock \emph{IEEE transactions on neural networks and learning systems},
  31\penalty0 (10):\penalty0 4394--4400, 2019.

\bibitem[Levy et~al.(2021)Levy, Kavis, and Cevher]{levy2021storm+}
Kfir Levy, Ali Kavis, and Volkan Cevher.
\newblock Storm+: Fully adaptive sgd with recursive momentum for nonconvex
  optimization.
\newblock \emph{NeurIPS}, 34:\penalty0 20571--20582, 2021.

\bibitem[Levy(2016)]{levy2016power}
Kfir~Y Levy.
\newblock The power of normalization: Faster evasion of saddle points.
\newblock \emph{arXiv preprint arXiv:1611.04831}, 2016.

\bibitem[Li(2022)]{li2022adaptive}
Xian Li.
\newblock Adaptive methods for parameter-agnostic nonconvex minimax
  optimization.
\newblock Master's thesis, ETH Zurich, Department of Computer Science, 2022.

\bibitem[Li and Orabona(2019)]{li2019convergence}
Xiaoyu Li and Francesco Orabona.
\newblock On the convergence of stochastic gradient descent with adaptive
  stepsizes.
\newblock In \emph{The 22nd international conference on artificial intelligence
  and statistics}, pages 983--992. PMLR, 2019.

\bibitem[Li and Orabona(2020)]{li2020high}
Xiaoyu Li and Francesco Orabona.
\newblock A high probability analysis of adaptive sgd with momentum.
\newblock \emph{arXiv preprint arXiv:2007.14294}, 2020.

\bibitem[Ling et~al.(2022)Ling, Sharp, and Jacobson]{ling2022vectoradam}
Selena Ling, Nicholas Sharp, and Alec Jacobson.
\newblock Vectoradam for rotation equivariant geometry optimization.
\newblock \emph{arXiv preprint arXiv:2205.13599}, 2022.

\bibitem[Liu et~al.(2020{\natexlab{a}})Liu, Jiang, He, Chen, Liu, Gao, and
  Han]{Liu_OnVarianceAdapt_2020}
Liyuan Liu, Haoming Jiang, Pengcheng He, Weizhu Chen, Xiaodong Liu, Jianfeng
  Gao, and Jiawei Han.
\newblock On the variance of the adaptive learning rate and beyond.
\newblock In \emph{International Conference on Learning Representations},
  2020{\natexlab{a}}.

\bibitem[Liu et~al.(2020{\natexlab{b}})Liu, Liu, Gao, Chen, and
  Han]{liu2020understanding}
Liyuan Liu, Xiaodong Liu, Jianfeng Gao, Weizhu Chen, and Jiawei Han.
\newblock Understanding the difficulty of training transformers.
\newblock \emph{arXiv preprint arXiv:2004.08249}, 2020{\natexlab{b}}.

\bibitem[Luo et~al.(2019)Luo, Xiong, Liu, and Sun]{luo2019adaptive}
Liangchen Luo, Yuanhao Xiong, Yan Liu, and Xu~Sun.
\newblock Adaptive gradient methods with dynamic bound of learning rate.
\newblock \emph{arXiv preprint arXiv:1902.09843}, 2019.

\bibitem[McMahan and Streeter(2010)]{mcmahan2010adaptive}
H~Brendan McMahan and Matthew Streeter.
\newblock Adaptive bound optimization for online convex optimization.
\newblock \emph{arXiv preprint arXiv:1002.4908}, 2010.

\bibitem[Moulines and Bach(2011)]{moulines2011non}
Eric Moulines and Francis Bach.
\newblock Non-asymptotic analysis of stochastic approximation algorithms for
  machine learning.
\newblock \emph{Advances in neural information processing systems}, 24, 2011.

\bibitem[Murray et~al.(2019)Murray, Swenson, and Kar]{murray2019revisiting}
Ryan Murray, Brian Swenson, and Soummya Kar.
\newblock Revisiting normalized gradient descent: Fast evasion of saddle
  points.
\newblock \emph{IEEE Transactions on Automatic Control}, 64\penalty0
  (11):\penalty0 4818--4824, 2019.

\bibitem[Nemirovski and Yudin(1983)]{nemirovski1983problem}
A.~S. Nemirovski and D.~B. Yudin.
\newblock \emph{Problem Complexity and Method Efficiency in Optimization.}
\newblock John Wiley UK/USA, 1983.

\bibitem[Nesterov(2003)]{nesterov2003introductory}
Yurii Nesterov.
\newblock \emph{Introductory lectures on convex optimization: A basic course},
  volume~87.
\newblock Springer Science \& Business Media, 2003.

\bibitem[Nesterov(1984)]{nesterov1984minimization}
Yurii~E Nesterov.
\newblock Minimization methods for nonsmooth convex and quasiconvex functions.
\newblock \emph{Matekon}, 29\penalty0 (3):\penalty0 519--531, 1984.

\bibitem[Palfinger(2022)]{palfinger2022continuous}
Werner Palfinger.
\newblock Continuous remeshing for inverse rendering.
\newblock \emph{Computer Animation and Virtual Worlds}, 33\penalty0
  (5):\penalty0 e2101, 2022.

\bibitem[Pascanu et~al.(2013)Pascanu, Mikolov, and
  Bengio]{pascanu2013difficulty}
Razvan Pascanu, Tomas Mikolov, and Yoshua Bengio.
\newblock On the difficulty of training recurrent neural networks.
\newblock In \emph{International conference on machine learning}, pages
  1310--1318. PMLR, 2013.

\bibitem[Polyak and Juditsky(1992)]{polyak1992acceleration}
Boris~T Polyak and Anatoli~B Juditsky.
\newblock Acceleration of stochastic approximation by averaging.
\newblock \emph{SIAM journal on control and optimization}, 30\penalty0
  (4):\penalty0 838--855, 1992.

\bibitem[Reddi et~al.(2018)Reddi, Kale, and Kumar]{reddi2018convergence}
Sashank~J Reddi, Satyen Kale, and Sanjiv Kumar.
\newblock On the convergence of adam and beyond.
\newblock In \emph{{{ICLR}}}, 2018.

\bibitem[Robbins and Monro(1951)]{robbins1951stochastic}
Herbert Robbins and Sutton Monro.
\newblock A stochastic approximation method.
\newblock \emph{The annals of mathematical statistics}, pages 400--407, 1951.

\bibitem[Schmidhuber(2015)]{schmidhuber2015deep}
J{\"u}rgen Schmidhuber.
\newblock Deep learning in neural networks: An overview.
\newblock \emph{Neural networks}, 61:\penalty0 85--117, 2015.

\bibitem[Shi and Li(2021)]{shi2021rmsprop}
Naichen Shi and Dawei Li.
\newblock Rmsprop converges with proper hyperparameter.
\newblock In \emph{International conference on learning representation}, 2021.

\bibitem[Streeter and McMahan(2010)]{streeter2010less}
Matthew Streeter and H~Brendan McMahan.
\newblock Less regret via online conditioning.
\newblock \emph{arXiv preprint arXiv:1002.4862}, 2010.

\bibitem[Sutton and Barto(2018)]{sutton-barto18}
Richard~S Sutton and Andrew~G Barto.
\newblock \emph{Reinforcement learning: An introduction}.
\newblock MIT press, 2018.

\bibitem[Wang et~al.(2021{\natexlab{a}})Wang, Meng, Chen, and
  Liu]{wang2021implicit}
Bohan Wang, Qi~Meng, Wei Chen, and Tie-Yan Liu.
\newblock The implicit bias for adaptive optimization algorithms on homogeneous
  neural networks.
\newblock In \emph{International Conference on Machine Learning}, pages
  10849--10858. PMLR, 2021{\natexlab{a}}.

\bibitem[Wang et~al.(2022)Wang, Zhang, Zhang, Meng, Ma, Liu, and
  Chen]{wang2022provable}
Bohan Wang, Yushun Zhang, Huishuai Zhang, Qi~Meng, Zhi-Ming Ma, Tie-Yan Liu,
  and Wei Chen.
\newblock Provable adaptivity in adam.
\newblock \emph{arXiv preprint arXiv:2208.09900}, 2022.

\bibitem[Wang et~al.(2021{\natexlab{b}})Wang, Magn{\'u}sson, and
  Johansson]{wang2021convergence}
Xiaoyu Wang, Sindri Magn{\'u}sson, and Mikael Johansson.
\newblock On the convergence of step decay step-size for stochastic
  optimization.
\newblock \emph{Advances in Neural Information Processing Systems},
  34:\penalty0 14226--14238, 2021{\natexlab{b}}.

\bibitem[Ward et~al.(2020)Ward, Wu, and Bottou]{ward2020adagrad}
Rachel Ward, Xiaoxia Wu, and Leon Bottou.
\newblock Adagrad stepsizes: {{Sharp}} convergence over nonconvex landscapes.
\newblock \emph{The Journal of Machine Learning Research}, 21\penalty0
  (1):\penalty0 9047--9076, 2020.

\bibitem[Xie et~al.(2022{\natexlab{a}})Xie, Zhou, Li, Lin, and
  Yan]{xie2022adan}
Xingyu Xie, Pan Zhou, Huan Li, Zhouchen Lin, and Shuicheng Yan.
\newblock Adan: Adaptive nesterov momentum algorithm for faster optimizing deep
  models.
\newblock \emph{arXiv preprint arXiv:2208.06677}, 2022{\natexlab{a}}.

\bibitem[Xie et~al.(2022{\natexlab{b}})Xie, Wang, Zhang, Sato, and
  Sugiyama]{xie2022adaptive}
Zeke Xie, Xinrui Wang, Huishuai Zhang, Issei Sato, and Masashi Sugiyama.
\newblock Adaptive inertia: Disentangling the effects of adaptive learning rate
  and momentum.
\newblock In \emph{International Conference on Machine Learning}, pages
  24430--24459. PMLR, 2022{\natexlab{b}}.

\bibitem[Yang et~al.(2022)Yang, Li, and He]{yang2022nest}
Junchi Yang, Xiang Li, and Niao He.
\newblock Nest your adaptive algorithm for parameter-agnostic nonconvex minimax
  optimization.
\newblock \emph{arXiv preprint arXiv:2206.00743}, 2022.

\bibitem[Zeiler(2012)]{Zeiler_2011_Adadelta}
D.~Matthew Zeiler.
\newblock Adadelta: An adaptive learning rate method.
\newblock \emph{arXiv preprint arXiv:1212.5701v1}, 2012.

\bibitem[Zhang et~al.(2019)Zhang, He, Sra, and Jadbabaie]{zhang2019gradient}
Jingzhao Zhang, Tianxing He, Suvrit Sra, and Ali Jadbabaie.
\newblock Why gradient clipping accelerates training: A theoretical
  justification for adaptivity.
\newblock In \emph{International Conference on Learning Representations}, 2019.

\bibitem[Zhang et~al.(2020)Zhang, Karimireddy, Veit, Kim, Reddi, Kumar, and
  Sra]{zhang2020adaptive}
Jingzhao Zhang, Sai~Praneeth Karimireddy, Andreas Veit, Seungyeon Kim, Sashank
  Reddi, Sanjiv Kumar, and Suvrit Sra.
\newblock Why are adaptive methods good for attention models?
\newblock \emph{Advances in Neural Information Processing Systems},
  33:\penalty0 15383--15393, 2020.

\bibitem[Zhang et~al.(2022)Zhang, Chen, Shi, Sun, and Luo]{zhang2022adam}
Yushun Zhang, Congliang Chen, Naichen Shi, Ruoyu Sun, and Zhi-Quan Luo.
\newblock Adam can converge without any modification on update rules.
\newblock \emph{Advances in neural information processing systems}, 2022.

\bibitem[Zhang(2018)]{zhang2018improved}
Zijun Zhang.
\newblock Improved adam optimizer for deep neural networks.
\newblock In \emph{2018 IEEE/ACM 26th international symposium on quality of
  service (IWQoS)}, pages 1--2. Ieee, 2018.

\bibitem[Zhao et~al.(2021)Zhao, Xie, and Li]{zhao2021convergence}
Shen-Yi Zhao, Yin-Peng Xie, and Wu-Jun Li.
\newblock On the convergence and improvement of stochastic normalized gradient
  descent.
\newblock \emph{Science China Information Sciences}, 64\penalty0 (3):\penalty0
  1--13, 2021.

\bibitem[Zhou et~al.(2018)Zhou, Chen, Cao, Tang, Yang, and
  Gu]{zhou2018convergence}
Dongruo Zhou, Jinghui Chen, Yuan Cao, Yiqi Tang, Ziyan Yang, and Quanquan Gu.
\newblock On the convergence of adaptive gradient methods for nonconvex
  optimization.
\newblock \emph{arXiv preprint arXiv:1808.05671}, 2018.

\bibitem[Zhou et~al.(2020)Zhou, Feng, Ma, Xiong, Hoi, et~al.]{zhou2020towards}
Pan Zhou, Jiashi Feng, Chao Ma, Caiming Xiong, Steven Chu~Hong Hoi, et~al.
\newblock Towards theoretically understanding why sgd generalizes better than
  adam in deep learning.
\newblock \emph{Advances in Neural Information Processing Systems},
  33:\penalty0 21285--21296, 2020.

\bibitem[Zou et~al.(2019)Zou, Shen, Jie, Zhang, and Liu]{zou2019sufficient}
Fangyu Zou, Li~Shen, Zequn Jie, Weizhong Zhang, and Wei Liu.
\newblock A sufficient condition for convergences of adam and rmsprop.
\newblock In \emph{Proceedings of the IEEE/CVF conference on computer vision
  and pattern recognition}, pages 11127--11135, 2019.

\end{thebibliography}
\bibliographystyle{plainnat}

\clearpage

\appendix

\section{Results Summary Table}
\label{sec:summary}
\begin{table*}[h]
		\centering
		\small
		\renewcommand{\arraystretch}{1.5}
  \hspace*{-3em}
  \noindent\adjustbox{max width=1.1\textwidth}{%
  		\begin{threeparttable}[b]
			\begin{tabular}{c | c | c | c | c}
				\hline
				\hline
				\multirow{2}{*}{\textbf{Algorithms}} & 
				\multirow{2}{*}{
					\makecell[c]{
						\textbf{Upper bound;} \vspace{0.1em}\\
						\textbf{deterministic} 
					}
				} & 
				\multirow{2}{*}{
					\makecell[c]{
						\textbf{Lower bound;} \vspace{0.1em}\\
						\textbf{deterministic}
					}
				} 
				&
                    \multirow{2}{*}{
					\makecell[c]{
						\textbf{Upper bound;} \vspace{0.1em}\\
						\textbf{stochastic}
					}
				} 
				&
                    \multirow{2}{*}{
					\makecell[c]{
						\textbf{Lower bound;} \vspace{0.1em}\\
						\textbf{stochastic}
					}
				} 
				\\
				& &  & &
				\\
				\hline 
				\hline
				\multirow{2}{*}{
                    \makecell[c]{
					    SGD (Alg. \ref{alg:sgd}) \vspace{0.1em}\\
						$\eta_t = \frac{\eta}{(t+1)^{\alpha}}$
					}
                    }
				& 
                    \multirow{2}{*}{
					\makecell[c]{
					    $\widetilde{\cO}\left( (4e^{2\alpha})^{\frac{(\eta\ell)^{1/\alpha}}{(1-\alpha)\wedge \alpha}}\epsilon^{\frac{-2}{(1-\alpha)\wedge \alpha}}\right)$ \vspace{-0.2em}\\
					 $\alpha \in (0,1)$	 \ $[$Thm.  \ref{thm:unbounded}, \ref{thm:general_stepsize}$]$
					}
				} 
				& 
                    \multirow{2}{*}{
					\makecell[c]{
						$\Omega\left( (8e)^{\eta^2\ell^2/8}\epsilon^{-4}\right)$ \vspace{0.1em}\\
					$\alpha = 1/2$ \	$[$Thm. \ref{thm:lower_bound}$]$
					}
				} 
				& 
                    \multirow{2}{*}{
					\makecell[c]{
						$\widetilde{\cO}\left( (4e^{2\alpha})^{\frac{(\eta\ell)^{1/\alpha}}{(1-\alpha)\wedge \alpha}}\epsilon^{\frac{-2}{(1-\alpha)\wedge \alpha}}\right)$  \vspace{-0.2em}\\
					   $\alpha \in (0,1)$ \ 	$[$Thm. \ref{thm:unbounded}, \ref{thm:general_stepsize}$]$
					}
				} 
				& 
                    \multirow{2}{*}{
					\makecell[c]{
						$\Omega\left( (8e)^ {\eta^2\ell^2/8}\epsilon^{-4}\right)$ \vspace{0.1em}\\
					$\alpha = 1/2$ \	$[$Thm. \ref{thm:lower_bound}$]$
					}
				} 
				\\
                    & & & &
                    \\
				\hline
				\multirow{2}{*}{
                    \makecell[c]{
						NSGD (Alg. \ref{alg:nsgd}) \vspace{0.1em}\\
						$\eta_t = \frac{\gamma_t}{\| g(x_t; \xi_t)\|}$
					}
                    }
				& 
                    \multirow{2}{*}{
					\makecell[c]{
					    $\widetilde{\cO}\left( \epsilon^{-2}\right)$, $\gamma_t = \frac{\gamma}{\sqrt{t+1}}$\vspace{0.1em}\\
         \citep{cutkosky2020momentum} \&
						$[$Prop. \ref{thm:NSGD}$]$
					}
				} 
				& 
                    \multirow{2}{*}{
					\makecell[c]{
						$\Omega\left( \epsilon^{-2}\right)$ \vspace{0.1em}\\
						\citep{carmon2020lower}
					}
				} 
				& 
                    \multirow{2}{*}{
                    \makecell[c]{
					N/A due to \\ lower bound
                 }
				} 
				& 
                    \multirow{2}{*}{
					\makecell[c]{
						Nonconvergent \vspace{0.1em}\\
						$\forall $ bounded $\{\gamma_t\}$ $[$Thm. \ref{lemma:nonconvergence}$]$
					}
				} 
				\\
                    & & & &
                    \\
				\hline
				\multirow{2}{*}{
                    \makecell[c]{
					    NSGD-M (Alg. \ref{alg:nsgdm}) \vspace{0.1em}\\
						$\eta_t = \frac{\gamma}{(t+1)^{\alpha}\| g_t\|}$
					}
                    }
				& 
                    \multirow{2}{*}{
					\makecell[c]{
					    $\widetilde{\cO}\left( \epsilon^{-2}\right)$, $\alpha = 1/2$ \vspace{0.1em}\\
         \citep{cutkosky2020momentum} \&
						$[$Prop. \ref{thm:NSGD}$]$
					}
				} 
				& 
                    \multirow{2}{*}{
					\makecell[c]{
						$\Omega\left( \epsilon^{-2}\right)$ \vspace{0.1em}\\
						\citep{carmon2020lower}
					}
				} 
				& 
                    \multirow{2}{*}{
					\makecell[c]{
						$\widetilde{\cO}\left( \epsilon^{-4}\right)$, $\alpha = 3/4$  \vspace{0.1em}\\
      \citep{cutkosky2020momentum} \&
						$[$Prop. \ref{alg:nsgdm}$]$
					}
				} 
				& 
                    \multirow{2}{*}{
					\makecell[c]{
						$\Omega\left( \epsilon^{-4}\right)$ \vspace{0.1em}\\
						\citep{arjevani2022lower}
					}
				} 
				\\
                    & & & &
                    \\
				\hline
				\multirow{3}{*}{
                    \makecell[c]{
					   AMSGrad-norm \vspace{0.3em}
                        (Alg. \ref{alg:amsgrad_full})\vspace{0.3em}\\
						 $\eta_t = \frac{\gamma}{(t+1)^{\alpha}\sqrt{\hat v^2_{t+1}}}$
					}
                    }
				& 
                    \multirow{3}{*}{
					\makecell[c]{
					    $\widetilde{\cO}\left( \epsilon^{-2/(1-\alpha)}\right)$, $\alpha \in (0, 1)$ \vspace{0.1em}\\
						$[$Thm. \ref{theorem:ams_grad}, \ref{thm:amsgrad_general_ss}$]$
					}
				} 
				& 
                    \multirow{3}{*}{
					\makecell[c]{
						$\Omega\left( \epsilon^{-2/(1-\alpha)}\right)$, $\alpha \in (0,1)$ \vspace{0.1em}\\
						Nonconvergent, $\alpha = 0$ \vspace{0.1em}\\
                        $[$Thm. \ref{thm:amsgrad_deter_lower_bound}$]$
					}
				} 
				& 
                    \multirow{3}{*}{
                    \makecell[c]{
					N/A due to \\ lower bound
                 }
				} 
				& 
                    \multirow{3}{*}{
					\makecell[c]{
						$\Omega\left( \epsilon^{-2/(1-\zeta)}\right)$, $\alpha = 1/2$ \vspace{0.3em}\\
					$\forall \zeta\in (0.5, 1)$ 	$[$Thm. \ref{thm:amsgrad_stoc}$]$
					}
				} 
				\\
                    & & & &
                    \\
                    & & & &
                    \\
				\hline
                    \multirow{2}{*}{
                    \makecell[c]{
					    AdaGrad-norm \vspace{0.1em}
						(Alg. \ref{alg:adagrad}) \\ $\eta_t = \frac{\eta}{\sqrt{ v_0^2 + \sum_{k = 0}^t \|g(x_k;\xi_k)\|^2}}$
					}
                    }
				& 
                    \multirow{2}{*}{
					\makecell[c]{
					    $\widetilde{\cO}\left( \epsilon^{-2}\right)$ \vspace{0.1em}\\
         \citep{yang2022nest} \&
						$[$Prop. \ref{prop:adagrad}$]$
					}
				} 
				& 
                    \multirow{2}{*}{
					\makecell[c]{
						$\Omega\left( \epsilon^{-2}\right)$  \vspace{0.1em}\\
						\citep{carmon2020lower}
					}
				} 
				& 
                    \multirow{2}{*}{
                    \makecell[c]{
						$\widetilde{\cO}\left( \epsilon^{-4}\right)$  \vspace{0.1em}\\
      \citep{yang2022nest} \&
						$[$Prop. \ref{prop:adagrad}$]$
					}
				} 
				& 
                    \multirow{2}{*}{
					\makecell[c]{
						$\Omega\left( \epsilon^{-4}\right)$ \vspace{0.1em}\\
					\citep{arjevani2022lower}
					}
				} 
				\\
                    & & & &
                    \\
				\hline
				\hline
			\end{tabular}
		\end{threeparttable}
}
  \caption{Comparisons of complexities to find an $\epsilon$-stationary point, i.e., $\mathbb{E} \|\nabla f(x)\| \leq \epsilon$, between SGD, NSGD, NSGD-M, AMSGrad-norm and AdaGrad-norm, with more general stepsize setups than Table~\ref{table:summary_results1}. We only assume $f$ is $\ell$-smooth, and unbiased stochastic gradients have bounded variance $\sigma$. Hyper-parameters (e.g., $\gamma$ and $\eta$) are arbitrary and untuned. In this table, $\cO$ and $\Omega$ hide polynomial terms in problem parameters and hyper-parameters. We use $\eta_t$ to denote the effective stepsize at iteration $t$.
  } 
		\label{table:summary_results2}
	\end{table*}

\section{Proofs for SGD in Section \ref{sec:sgd}}
\label{sec:apdx_Sgd}
\bigskip

\subsection{Upper Bounds for SGD}
We provide an extended theorem of Theorem \ref{thm:unbounded} and include more general decaying stepsizes $\eta_t = \eta/(t+1)^\alpha$ with $0 < \alpha < 1$. 

\begin{theorem}
\label{thm:general_stepsize}
Under Assumptions \ref{assume:smoothness} and \ref{assume:stochastic_grad}, if we run SGD with stepsize $\eta_t = \eta / (t + 1)^{\alpha}$ where $\eta > 0$ and $1/2 \leq \alpha < 1$, then with $\eta \leq 1/\ell$, 
\begin{equation*}
\begin{split}
\Ep{\frac{1}{T}\sum_{t=0}^{T-1} \norm*{\nabla f(x_t)}^2}  \leq \left\{
\begin{aligned}
    &\frac{2}{\eta\sqrt{T}} \left(\Delta +  \frac{\ell\sigma^2\eta^2}{2}(1+\log T)\right), \text{ when } \alpha = 1/2, \\
    &\frac{2}{\eta T^{1-\alpha}} \left(\Delta +  \frac{\ell\sigma^2\eta^2}{2\left(1-2^{1-2\alpha} \right)}\right), \text{ when } 1/2 < \alpha  < 1, \\
    &\frac{2}{\eta T^{\alpha}} \left(\frac{\Delta}{T^{1-2\alpha}} +  \frac{\ell\sigma^2\eta^2}{2\left(1 - 2\alpha \right)}\right), \text{ when } 0 < \alpha  < 1/2;
\end{aligned}
\right.
\end{split}
\end{equation*}
with $\eta > 1/\ell$, 
\begin{equation*}
\begin{split}
\Ep{\frac{1}{T}\sum_{t=0}^{T-1} \norm*{\nabla f(x_t)}^2}  \leq \left\{ 
\begin{aligned}
&\frac{\sqrt{2} \ (4e)^{\tau}}{\eta\sqrt{\pi\tau T}}\left[1 + \ell\eta\left(1+2\sqrt{\tau}\right) \right]\left( \Delta + \frac{\ell\sigma^2\eta^2}{2}(1+\log T)\right),  \\ & \qquad \text{ when } \alpha = 1/2, \\
&\frac{2 \ (4e^{2\alpha})^{\tau}\ }{\eta(2\pi\tau)^\alpha \ T^{1-\alpha}}\left[1 + \ell\eta\left(1 + \frac{\tau^{1-\alpha} }{1-\alpha}\right) \right]\left( \Delta + \frac{\ell\sigma^2\eta^2}{2\left(1-2^{1-2\alpha} \right)}\right), \\ & \qquad  \text{ when } 1/2 < \alpha < 1,  \\
&\frac{2 \ (4e^{2\alpha})^{\tau}\ }{\eta(2\pi\tau)^\alpha \ T^{\alpha}}\left[1 + \ell\eta\left(1 + \frac{\tau^{1-\alpha} }{1-\alpha}\right) \right]\left( \frac{\Delta}{T^{1-2\alpha}} + \frac{\ell\sigma^2\eta^2}{2\left(1-2\alpha \right)}\right),  \\ & \qquad \text{ when } 0 < \alpha < 1/2,
\end{aligned}
\right.
\end{split}
\end{equation*}
where  $\tau = \lceil (\eta\ell)^{1/\alpha} - 1\rceil$.
\end{theorem}

\begin{proof}
 By $\ell$-smoothness of $f(\cdot)$,
\begin{align*}
    f(x_{t+1}) & \leq f(x_t) + \langle \nabla f(x_t), x_{t+1} - x_t \rangle + \frac{\ell}{2}\|x_{t+1}- x_t \|^2 \\
    & = f(x_t) - \eta_t \langle \nabla f(x_t), g(x_t;\xi_t) \rangle + \frac{\ell\eta_t^2}{2} \|g(x_t;\xi_t)\|^2
\end{align*}
Taking expectation, 
\begin{align} \nonumber
    \mathbb{E}  f(x_{t+1}) & \leq  \mathbb{E} f(x_t) - \eta_t \mathbb{E} \|\nabla f(x_t)\|^2 + \frac{\ell\eta_t^2}{2} \mathbb{E}\|\nabla f(x_t)\|^2 + \frac{\ell\eta_t^2}{2} \sigma^2 \\
    & \leq  \mathbb{E} f(x_t) - \left(\eta_t - \frac{\ell\eta_t^2}{2}\right) \mathbb{E} \|\nabla f(x_t)\|^2 + \frac{\ell\eta_t^2}{2}\sigma^2
\end{align}
We note that $\eta_t - \frac{\ell\eta_t^2}{2} \geq \frac{\eta_t}{2}$ when $\eta_t \leq \frac{1}{\ell}$, i.e., $t \geq (\eta \ell)^{1/\alpha} -1$. Define $\tau = \lceil (\eta\ell)^{1/\alpha} - 1 \rceil$. Therefore,
for all $t < \tau$, 
\begin{align} 
\label{before_t}
    \mathbb{E}  f(x_{t+1})  \leq  \mathbb{E} f(x_t) +  \frac{\ell\eta_t^2}{2} \mathbb{E}\|\nabla f(x_t)\|^2 + \frac{\ell\eta_t^2}{2}\sigma^2. 
\end{align}
For all $t \geq \tau$, we have 
\begin{align} 
    \mathbb{E}  f(x_{t+1})  \leq  \mathbb{E} f(x_t) -  \frac{\eta_t}{2} \mathbb{E} \|\nabla f(x_t)\|^2 + \frac{\ell\eta_t^2}{2}\sigma^2. 
\end{align}
Summing from $t = \tau$ to $ T - 1$, we have 
\begin{align} \label{second_grad_sum}
    \sum_{t = \tau}^{ T -1} \frac{\eta_t}{2} \mathbb{E} \|\nabla f(x_t)\|^2 & \leq \mathbb{E} f(x_{\tau}) - \mathbb{E} f(x_{ T}) +  \sum_{t = \tau}^{ T -1}\frac{\ell\eta_t^2}{2}\sigma^2 %
\end{align}
Now we want to bound $\mathbb{E} f(x_{\tau}) - f(x_{T}) \leq \mathbb{E} f(x_{\tau}) - f^*$, where $f^* \triangleq \min_{x\in \mathbb{R}^d}f(x)$. From (\ref{before_t}), 
\begin{align} \nonumber
    \mathbb{E}  f(x_{t+1}) - f^* & \leq  \mathbb{E} f(x_t) - f^*  + \frac{\ell\eta_t^2}{2} \mathbb{E}\|\nabla f(x_t)\|^2 + \frac{\ell\eta_t^2}{2} \sigma^2 \\ \nonumber
    & \leq  (1 + \ell^2\eta_t^2)[\mathbb{E} f(x_t) - f^* ] + \frac{\ell\eta_t^2}{2}\sigma^2, %
\end{align}
where in the second inequality we use $\|\nabla f(x)\|^2 \leq 2\ell[f(x) - f^*]$. When $\tau = 0$, $f(x_\tau) - f(x_T) \leq \Delta$; when $\tau \geq 1$, recursing the inequality above, for $j \leq \tau$,
\begin{align} \nonumber
    \mathbb{E} f(x_j) - f^* & \leq
    \Delta\left(\prod_{t = 0}^{j-1}1 + \ell^2\eta_t^2\right)  + \sum_{k=0}^{j - 2} \left( \prod_{t = k+1}^{j - 1}1 + \ell^2\eta_t^2\right)\frac{\ell\eta_k^2}{2}\sigma^2 + \frac{\ell\eta^2_{j-1}}{2}\sigma^2 \\ \nonumber
    & \leq 
    \left(\prod_{t = 0}^{j-1}1 +\ell^2\eta_t^2\right)\left( \Delta + \sum_{t=0}^{j-1}\frac{\ell\eta_t^2}{2}\sigma^2\right) \\ \label{before_tau}
    & \leq \left(\prod_{t = 0}^{\tau-1}1 +\ell^2\eta_t^2\right)\left( \Delta + \sum_{t=0}^{\tau-1}\frac{\ell\eta_t^2}{2}\sigma^2\right).
\end{align}
Also, with $\|\nabla f(x)\|^2 \leq 2\ell[f(x) - f^*]$, if $\tau \geq 1$, 
\begin{align*}
    \sum_{t = 0}^{\tau -1} \frac{\eta_t}{2} \mathbb{E} \|\nabla f(x_t)\|^2 &  \leq  \sum_{t = 0}^{\tau -1} \eta_t \ell \mathbb{E} \left( f(x_t) - f^*\right)  \\
    & \leq  \ell \left(\sum_{t = 0}^{\tau -1} \eta_t\right) \left(\prod_{t = 0}^{\tau-1}1 +\ell^2\eta_t^2\right)\left( \Delta + \sum_{t=0}^{\tau-1}\frac{\ell\eta_t^2}{2}\sigma^2\right),
\end{align*}
where in the second inequality we use (\ref{before_tau}) . Combining with (\ref{second_grad_sum}) and (\ref{before_tau}), if $\tau \geq 1$
\begin{align*}
 \sum_{t = 0}^{T -1} \frac{\eta_t}{2} \mathbb{E} \|\nabla f(x_t)\|^2 
& \leq \left(\prod_{t = 0}^{\tau-1}1 +\ell^2\eta_t^2\right)\left( \Delta + \sum_{t=0}^{\tau-1}\frac{\ell\eta_t^2}{2}\sigma^2\right) +  \sum_{t = 0}^{T -1}\frac{\ell\eta_t^2}{2}\sigma^2 \\ 
& \qquad + \ell \left(\sum_{t = 0}^{\tau -1} \eta_t\right)  \left(\prod_{t = 0}^{\tau-1}1 +\ell^2\eta_t^2\right)\left( \Delta + \sum_{t=0}^{\tau-1}\frac{\ell\eta_t^2}{2}\sigma^2\right).
\end{align*}
We note that 
\begin{align*}
      \prod_{t = 0}^{\tau-1}\left(1 +\ell^2\eta_t^2 \right) &= \prod_{t = 0}^{\tau-1}\left(1 + \frac{\ell^2\eta^2}{(t+1)^{2\alpha}} \right) = \frac{\prod_{t= 0}^{\tau - 1}\left(\ell^2\eta^2 + (t + 1)^{2\alpha}\right)}{(\tau!)^{2\alpha}}  \leq \frac{\left(\ell^2\eta^2 + \tau^{2\alpha} \right)^\tau}{(\tau!)^{2\alpha}} \\
      & \leq \frac{\left( 2\ell^2\eta^2\right)^\tau}{(\tau!)^{2\alpha}} 
      \leq \frac{(2\ell^2\eta^2)^\tau}{\left[\sqrt{2\pi\tau}\left( \frac{\tau}{e}\right)^\tau\exp{(\frac{1}{12\tau+1})}\right]^{2\alpha}} \\
      & \leq \frac{1}{(2\pi\tau)^{\alpha}}\left( \frac{2\ell^2\eta^2 e^{2\alpha}}{\tau^{2\alpha}}\right)^\tau \leq \frac{1}{(2\pi\tau)^\alpha}\left(4e^{2\alpha}\right)^\tau,
\end{align*}
where in the third inequality we use Stirling's approximation. 
Therefore, 
\begin{equation*}
    \sum_{t = 0}^{T -1} \frac{\eta_t}{2} \mathbb{E} \|\nabla f(x_t)\|^2 \leq \frac{1}{(2\pi\tau)^\alpha}\left(4e^{2\alpha}\right)^\tau\left[1 + \ell\left(\sum_{t = 0}^{\tau -1} \eta_t\right) \right]\left( \Delta + \sum_{t=0}^{T-1}\frac{\ell\eta_t^2}{2}\sigma^2\right).
\end{equation*}
Plugging in $\eta_t = \eta/(t+1)^\alpha$, when $\alpha = 1/2$,
\begin{equation*}
    \sum_{t = 0}^{T -1} \mathbb{E} \|\nabla f(x_t)\|^2 
    \leq \frac{\sqrt{2T}}{\eta\sqrt{\pi\tau}}(4e)^{\tau}\left[1 + \ell\eta\left(1+2\sqrt{\tau}\right) \right]\left( \Delta + \frac{\ell\sigma^2\eta^2}{2}(1+\log T)\right);
\end{equation*}
when $1/2 < \alpha < 1$,
\begin{equation*}
    \sum_{t = 0}^{T -1} \mathbb{E} \|\nabla f(x_t)\|^2 
    \leq \frac{2\ T^{\alpha}}{\eta(2\pi \tau)^{\alpha}}(4e^{2\alpha})^{\tau}\left[1 + \ell\eta\left(1 + \frac{\tau^{1-\alpha} }{1-\alpha}\right) \right]\left( \Delta + \frac{\ell\sigma^2\eta^2}{2\left(1-2^{1-2\alpha} \right)}\right).
\end{equation*}
when $0 < \alpha < 1/2$,
\begin{equation*}
    \sum_{t = 0}^{T -1} \mathbb{E} \|\nabla f(x_t)\|^2 
    \leq \frac{2\ T^{\alpha}}{\eta(2\pi \tau)^{\alpha}}(4e^{2\alpha})^{\tau}\left[1 + \ell\eta\left(1 + \frac{\tau^{1-\alpha} }{1-\alpha}\right) \right]\left( \Delta + \frac{\ell\sigma^2\eta^2 T^{1-2\alpha}}{2\left(1-2\alpha \right)}\right);
\end{equation*}
If $\tau = 0$, from (\ref{second_grad_sum}), 
\begin{equation*}
    \sum_{t = 0}^{ T -1} \frac{\eta_t}{2} \mathbb{E} \|\nabla f(x_t)\|^2  \leq \Delta +  \sum_{t = 0}^{ T -1}\frac{\ell\eta_t^2}{2}\sigma^2,
\end{equation*}
Plugging in $\eta_t$, when $\alpha = 1/2$,
\begin{equation*}
    \sum_{t = 0}^{ T -1} \mathbb{E} \|\nabla f(x_t)\|^2  \leq \frac{2\sqrt{T}}{\eta} \left(\Delta +  \frac{\ell\sigma^2\eta^2}{2}(1+\log T)\right);
\end{equation*}
when $1/2 < \alpha < 1$,
\begin{equation*}
    \sum_{t = 0}^{ T -1} \mathbb{E} \|\nabla f(x_t)\|^2  \leq \frac{2T^\alpha}{\eta} \left(\Delta +  \frac{\ell\sigma^2\eta^2}{2\left(1-2^{1-2\alpha} \right)}\right).
\end{equation*}
when $0 < \alpha < 1/2$,
\begin{equation*}
    \sum_{t = 0}^{ T -1} \mathbb{E} \|\nabla f(x_t)\|^2  \leq \frac{2T^\alpha}{\eta} \left(\Delta +  \frac{\ell\sigma^2\eta^2 T^{1 - 2\alpha}}{2\left(1- 2\alpha \right)}\right).
\end{equation*}

\end{proof}

\begin{remark}
When we run SGD with stepsize $\eta_t = \eta/(t+1)^{\alpha}$, where $1/2 < \alpha < 1$, Theorem \ref{thm:general_stepsize} implies a complexity of $\cO\left((4e^{2\alpha})^{\frac{(\eta\ell)^{1/\alpha}}{1-\alpha}}(\eta\ell)^{\frac{1}{\alpha(1-\alpha)}} \cdot \epsilon^{\frac{-2}{1-\alpha}}\right)$ in the large initial stepsize regime $\eta > 1/\ell$. Compared with the case $\alpha = 1/2$, when $\alpha$ is larger, the convergence rate in $T$ is slower, but it also comes with a smaller exponent, i.e., $(\eta\ell)^{1/\alpha}$. This is because $\alpha = 1/2$ leads to the best convergence rate  in $T$ \citep{drori2020complexity}, while the faster decaying stepsize $\alpha > 1/2$ will reach the desirable stepsize $1/\ell$ earlier so that it accumulates less gradient norms before $\tau$. For $0 < \alpha < 1/2$, however, it comes with both a larger exponent and a slower convergence rate. 

\end{remark}

\begin{proposition} [with bounded gradient]
\label{lemma:sgd_with_bound_g}
Under Assumption \ref{assume:smoothness}, \ref{assume:stochastic_grad} and additionally assuming that the gradient norm is upper bounded by $G$, i.e., $\|\nabla f(x) \| \leq G$ for all $x \in \bR^{d}$, if we run SGD with stepsize $\eta_t = \eta / \sqrt{t + 1}$ with $\eta > 0$, then
\begin{align*}
  \Ep{\frac{1}{T} \sum_{t=0}^{T-1} \norm*{\nabla f(x_t)}^2} 
  \leq \frac{1}{\sqrt{T}}\left(\frac{\Delta}{\eta} + \frac{\ell \eta \left(G^2 + \sigma^2\right)}{2} \log T \right).
\end{align*}
\end{proposition}

\begin{proof}
By the smoothness of $f(\cdot)$, we have
\begin{align*}
  f(x_{t+1}) \leq f(x_t) - \eta_t \inp*{\nabla F(x_t; \xi_t)}{\nabla f(x_t)} + \frac{\ell\eta_t^2}{2} \norm*{\nabla F(x_t; \xi_t)}^2.
\end{align*}
Taking expectation and  summing from $t=0$ to $T-1$,
\begin{align*}
 \Ep{\sum_{t=0}^{T-1} \eta_t \norm*{\nabla f(x_t)}^2}
 &\leq f(x_0) - f(x_{T}) + \frac{\ell}{2}\sum_{t=0}^{T-1} \eta_t^2 \Ep{\norm*{\nabla F(x_t; \xi_t)}^2} \\
 &\leq \Delta + \frac{\ell}{2}\sum_{t=0}^{T-1} \eta_t^2 \Ep{\norm*{\nabla F(x_t; \xi_t)}^2} \\
 &\leq \Delta + \frac{\ell}{2}\sum_{t=0}^{T-1} \eta_t^2 \left( \norm*{\nabla f(x_t)}^2 + \Ep{\norm*{\nabla F(x_t; \xi_t) - f(x_t)}^2} \right) \\
 &\leq \Delta + \frac{\ell}{2}\sum_{t=0}^{T-1} \eta_t^2 \left(G^2 + \sigma^2\right).
\end{align*}
Let $\eta_t = \eta/\sqrt{t+1}$, 
\begin{align*}
  \frac{\eta}{\sqrt{T}} \Ep{\sum_{t=0}^{T-1} \norm*{\nabla f(x_t)}^2} \leq \Ep{\sum_{t=0}^{T-1} \frac{\eta}{\sqrt{t+1}} \norm*{\nabla f(x_t)}^2} & \leq \Delta + \frac{\ell \left(G^2 + \sigma^2\right) }{2}\sum_{t=0}^{T-1} \frac{\eta^2}{t+1}, \\
  & \leq \Delta + \frac{\ell \eta^2 \left(G^2 + \sigma^2\right) }{2}\log T.
\end{align*}

\end{proof}

\subsection{Lower Bound for SGD}

\paragraph{Proof for Theorem \ref{thm:lower_bound}}
\mbox{}\\

\begin{proof} 
We construct the hard instance with 4 segments of quadratic functions. The function is symmetric about $x = 0$, and we will define it on $x \leq 0$ as below. We illustrate it in Figure \ref{fig:lower_bound_demo}. \\
\textit{Segment 1.} We define 
$
f(x) = \frac{\ell}{2} x^2.
$
We pick $x_0$ such that $f(x_0) - f^* = \Delta$, i.e., $x_0 = \sqrt{\frac{2\Delta}{\ell}}$. We define $t_0$ to be the first iteration that $\eta_{t_0} = \frac{\eta}{\sqrt{t_0 + 1}} \geq \frac{4}{\ell}$, i.e., $t_0 = \left \lfloor \frac{\eta^2\ell^2}{16} - 1  \right\rfloor$. With the update rule $
x_{t+1} = x_t - \frac{\eta\ell}{\sqrt{t+1}}x_t = \left(1 - \frac{\eta\ell}{\sqrt{t+1}} \right)x_t
$,
we have for $t \leq t_0$
\begin{align*}
|x_t|^2 & = \left[\prod_{k=1}^t
\left(\frac{\eta\ell}{\sqrt{k}} -1 \right)
\right]^2|x_0|^2 
 \geq \prod_{k=1}^t  \frac{\eta^2\ell^2}{2k}|x_0|^2 \\
& = \frac{(\eta^2\ell^2/2)^{t}}{(t)!}|x_0|^2  > \frac{(\eta^2\ell^2/2)^t}{\sqrt{2\pi t}(t/e)^te^{1/12t}}|x_0|^2 
 \geq \frac{1}{3\sqrt{t}}(8e)^t|x_0|^2,
\end{align*}
where in the inequality we use $\left(\frac{\eta\ell}{\sqrt{k}} - 1\right)^2 \geq \frac{\eta^2\ell^2}{2k}$ with $k\leq t_0$, in the second inequality we use Stirling's approximation, and in the last inequality we use $t\leq \eta^2\ell^2/16$. We note that
$$
|x_{t_0}|^2 \geq \frac{1}{3\sqrt{t_0}}(8e)^{t_0}|x_0|^2 \geq \frac{4}{3\eta\ell}(8e)^{\eta^2\ell^2/16-2}|x_0|^2.
$$
Without loss of generality, we assume $x_{t_0} > 0$.
Segment 1 is define on the domain $\{x: |x| \leq x_{t_0}\}$. \vspace{5mm} \\ 
\textit{Segment 2.} This segment is the mirror of Segment 1. On domain $\{x: x_{t_0} \leq x \leq 2x_{t_0}\}$, we define
$
f(x) = -\frac{\ell}{2}(x- 2x_{t_0})^2 + \ell x_{t_0}^2.
$
\vspace{5mm} \\
\textit{Segment 3. } We note that 
$$
x_{t_0+1} = x_{t_0} - \eta_{t_0}\ell x_{t_0} = \left(1 - \frac{\eta\ell}{\sqrt{t_0 + 1}} \right) x_{t_0} \leq -3x_{t_0},
$$
where  the inequality is from the definition of $t_0$, and   
$$\Tilde{\Delta} \triangleq 
\frac{\ell x_{t_0}^2}{2} \geq
\frac{2}{3\eta}(8e)^{\eta^2\ell^2/16-2}|x_0|^2 = \frac{4}{3\eta\ell}(8e)^{\eta^2\ell^2/16-2}\Delta.
$$
We construct a quadratic function such that: it passes $(-2x_{t_0}, \ell x_{t_0}^2)$ with gradient 0; the gradient at $x = x_{t_0 + 1}$ is $\frac{\sqrt{\widetilde \Delta}}{2 \sqrt{\max\{1/\ell, \sum_{t = t_0 + 1}^{T-1} \eta_t \}}}$. This quadratic function is uniquely defined to be 
$$
f(x) = -\frac{\sqrt{\widetilde \Delta}\ (x + 2x_{t_0})^2}{4(-2x_{t_0} - x_{t_0+1}) \sqrt{\max\{1/\ell, \sum_{t = t_0 + 1}^{T-1} \eta_t \}}} + \ell x_{t_0}^2.
$$
It can be verified that this function is $\ell$-smooth: as $x_{t_0+1} \leq -3x_{t_0}$,
\begin{align*}
    \frac{\sqrt{\frac{1}{2}\ell x_{t_0}^2}}{2(-2x_{t_0} - x_{t_0+1})} \leq \sqrt{\ell}
    \Longleftrightarrow \frac{\sqrt{\Tilde{\Delta}}}{2(-2x_{t_0} - x_{t_0+1})\sqrt{\ell}}  \leq \ell \\
    \Longrightarrow \frac{\sqrt{\widetilde \Delta}\ }{2(-2x_{t_0} - x_{t_0+1}) \sqrt{\max\{1/\ell, \sum_{t = t_0 + 1}^{T-1} \eta_t \}}} \leq \ell.
\end{align*}
The function is defined on the domain $\{x: x_{t_0+1} \leq x \leq -2x_{t_0} \}$. %
\vspace{5mm} \\
\textit{Segment 4. } For convenience, we define $w = f(x_{t_0+1})$. We can verify that $w \geq \Tilde{\Delta}$: as $\frac{1}{\sqrt{t_0}} < \frac{4}{\eta\ell}$,
\begin{align*}
   & -2x_{t_0} - \left(1- \frac{\eta\ell}{\sqrt{t_0 + 1}} \right)x_{t_0} \leq 4\sqrt{\frac{x_{t_0}}{2}} \Longleftrightarrow \frac{-2x_{t_0} - x_{t_0+1}}{4} \leq \sqrt{\frac{x_{t_0}}{2}} \\
   & \Longrightarrow \frac{-2x_{t_0} - x_{t_0+1}}{4 \sqrt{\max\{1/\ell, \sum_{t = t_0 + 1}^{T-1} \eta_t \}}} \leq \sqrt{\frac{1}{2}\ell x_{t_0}^2} 
    \Longleftrightarrow \frac{\sqrt{\widetilde \Delta} \ (x_{t_0+1} + 2x_{t_0})^2 }{4(-2x_{t_0} - x_{t_0+1}) \sqrt{\max\{1/\ell, \sum_{t = t_0 + 1}^{T-1} \eta_t \}}} \leq \Tilde{\Delta}.
\end{align*}
So we conclude $w \geq \Tilde{\Delta}$. Now we construct a quadratic function similar to that in Proposition 1 of \citep{drori2020complexity}: it passes $(x_{t_0+1}, w)$ with gradient $\frac{\sqrt{\widetilde \Delta}}{2\sqrt{ \max\{1/\ell, \sum_{t = t_0 + 1}^{T-1} \eta_t \}}}$; 
the minimum is at $x = x_{t_0+1} - \sqrt{\Tilde{\Delta}\max\{1/\ell, \sum_{t = t_0 + 1}^{T-1} \eta_t \}} $. This quadratic function is  defined to be
\begin{equation*}
    f(x) = \frac{\left(x - x_{t_0+1} + \sqrt{\Tilde{\Delta}\max\{1/\ell, \sum_{t = t_0 + 1}^{T-1} \eta_t \}} \right)^2}{4 \max\{1/\ell, \sum_{t = t_0 + 1}^{T-1} \eta_t \}} + w - \frac{\Tilde{\Delta}}{4}
\end{equation*}
on the domain $\{x: x \leq x_{t_0+1} \}$. It is obvious that $f(x) \geq 0$ and is $\ell$-smooth. Following the same reasoning of Proposition 1 in \citep{drori2020complexity}, also presented as Lemma \ref{lemma:lower_bound_help} in the appendix for completeness, we can conclude for all $t: t_0 + 1 \leq t \leq T$,
\begin{align*}
    |\nabla f(x_t)|  \geq \frac{\sqrt{\Tilde{\Delta}}}{4\sqrt{\max\{1/\ell, \sum_{t = t_0 + 1}^{T-1} \eta_t \}}} 
     \geq \frac{1}{4}\sqrt{\Tilde{\Delta}}\min\bigg\{\sqrt{\ell}, (2\eta)^{-1/2}T^{-1/4} \bigg\},
\end{align*}
where in the second inequality we use 
$\sum_{t = t_0 + 1}^{T-1} \eta_t = \sum_{t = t_0 + 1}^{T-1} \frac{\eta}{\sqrt{t+1}} \leq 2\eta T^{1/2} $.

\end{proof}

The following lemma is used in the proof of Theorem \ref{thm:lower_bound}. It is a straightforward modification of Proposition 1 in \citep{drori2020complexity}. We present it here for completeness. 

\begin{lemma}
\label{lemma:lower_bound_help}
Under the same setting and notations as the proof of 
Theorem \ref{thm:lower_bound}, if we run gradient descent with stepsize $\{\eta_t\}_{t=t_0+1}^{T-1}$ starting from point $x_{t_0+1}$ on function
\begin{equation*}
    f(x) = \frac{\left(x - x_{t_0+1} + \sqrt{\Tilde{\Delta}\max\{1/\ell, \sum_{t = t_0 + 1}^{T-1} \eta_t \}} \right)^2}{4 \max\{1/\ell, \sum_{t = t_0 + 1}^{T-1} \eta_t \}} + w - \frac{\Tilde{\Delta}}{4},
\end{equation*}
 then for all $t: t_0 + 1 \leq t \leq T$,
\begin{align*}
    |\nabla f(x_t)| & \geq \frac{\sqrt{\Tilde{\Delta}}}{4\sqrt{\max\{1/\ell, \sum_{t = t_0 + 1}^{T-1} \eta_t \}}}.
\end{align*}
\end{lemma}

\begin{proof}
From the update of gradient descent, we have
\begin{align*}
    x_{t+1} & = x_t - \eta_t \cdot \frac{x_t - x_{t_0+1} + \sqrt{\Tilde{\Delta}\max\{1/\ell, \sum_{t = t_0 + 1}^{T-1} \eta_t \}} }{2 \max\{1/\ell, \sum_{t = t_0 + 1}^{T-1} \eta_t \}},
\end{align*}
which leads to 
\begin{align*}
    & \fakeeq x_{t+1}  - x_{t_0+1} + \sqrt{\Tilde{\Delta}\max\{1/\ell, \sum_{t = t_0 + 1}^{T-1} \eta_t \}} \\ & = \left( 1 -  \frac{\eta_t}{2 \max\{1/\ell, \sum_{t = t_0 + 1}^{T-1} \eta_t \}}\right) \left(  x_{t}  - x_{t_0+1} + \sqrt{\Tilde{\Delta}\max\{1/\ell, \sum_{t = t_0 + 1}^{T-1} \eta_t \}}\right).
\end{align*}
Recursing this, for $j\leq T$
\begin{align*}
    & \fakeeq x_{j}  - x_{t_0+1} + \sqrt{\Tilde{\Delta}\max\{1/\ell, \sum_{t = t_0 + 1}^{T-1} \eta_t \}}
    \\
     & = \prod_{k=t_0+1}^{j-1} \left( 1 -  \frac{\eta_k}{2 \max\{1/\ell, \sum_{t = t_0 + 1}^{T-1} \eta_t \}}\right) \left(  x_{t_0+1}  - x_{t_0+1} + \sqrt{\Tilde{\Delta}\max\{1/\ell, \sum_{t = t_0 + 1}^{T-1} \eta_t \}}\right) \\
    & \geq \exp\left(\log \frac{1}{2}\cdot \sum_{k=t_0+1}^{j-1} \frac{\eta_k}{2 \max\{1/\ell, \sum_{t = t_0 + 1}^{T-1} \eta_t \}} \right)  \sqrt{\Tilde{\Delta}\max\{1/\ell, \sum_{t = t_0 + 1}^{T-1} \eta_t \}} \\
     & \geq \frac{1}{2}\sqrt{\Tilde{\Delta}\max\{1/\ell, \sum_{t = t_0 + 1}^{T-1} \eta_t \}},
\end{align*}
where in the second inequality, we use $1-z/2 \geq \exp(\log \frac{1}{2}\cdot z)$ for $0\leq z\leq 1$. This directly implies what we want to prove by computing $\nabla f(x_j)$.

\end{proof}

\section{Proofs for NSGD Family in Section \ref{sec:adaptive}}
\label{sec:apdx_adaptive}

\bigskip

\begin{algorithm}[ht] 
    \caption{Normalized Stochastic Gradient Descent (NSGD)}
    \setstretch{1.23}
    \begin{algorithmic}[1]
      \STATE \textbf{Input:} initial point $x_0$
        \FOR{$t = 0,1,2,...$}
            \STATE sample  $\xi_t$ %
            \STATE  $x_{t+1} = x_t -  \frac{\gamma_t}{\|g(x_t; \xi_t)\|}g(x_t; \xi_t)$
        \ENDFOR
    \end{algorithmic} \label{alg:nsgd}
\end{algorithm}

\begin{figure}[t]
    \centering
    \subfigure[example 1]{
      \centering
      \includegraphics[width=0.24\linewidth]{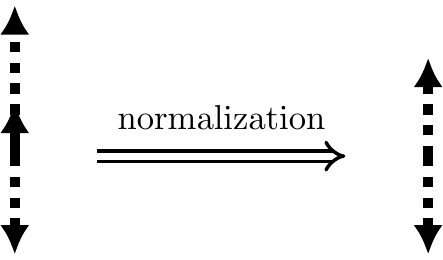}
    }
    \hspace{3cm}
    \subfigure[example 2]{
      \centering
      \includegraphics[width=0.3\linewidth]{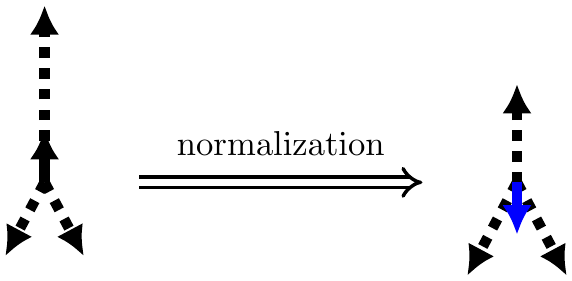}
    }
    \caption{The expected update of NSGD can vanish (example 1) or
    be in the opposite direction (example 2) of the true gradient. The solid black arrow represents
    the true gradient and the dashed arrows are the possible stochastic gradients (with equal possibilities). The solid blue arrow is the expected direction of NSGD update.}
    \label{fig:nsgd_diverge}
\end{figure}

\paragraph{Proof for Theorem \ref{lemma:nonconvergence}} \mbox{}\\
\begin{proof}
 Let us pick $f(x) = \frac{L}{2}x^2$ with $\frac{\epsilon^2}{2\Delta} < L \leq\ell$ and $L < \frac{\sigma - \epsilon}{\gamma_{\max}}$. Then we pick $x_0$ such that $\frac{\epsilon}{L} < x_0 < \sqrt{\frac{2\Delta}{L}} $, which implies that $\|\nabla f(x_0)\| > \epsilon$ and $f(x_0) - \min_x f\leq \Delta$. Now we define $D = \{x: -w \leq x \leq w  \}$ with $\frac{\epsilon}{L} + \gamma_{\max} < w < \frac{\sigma}{L}$. For $x \in D$, we have $\|\nabla f(x)\| \leq \sigma$ and we construct the noisy gradients: with $\delta > 1$
 $$
g(x; \xi) = (1 + \delta) \nabla f(x) \text{ w.p. } \frac{1}{2}, \text{ and } g(x; \xi) = (1 - \delta)\nabla f(x)  \text{ w.p. } \frac{1}{2}.
$$
It is obvious that $\nabla f(x)=\mathbb{E}[g(x; \xi)]$ and the variance at this point $\mathbb{E}\|\nabla f(x) - g(x; \xi)\|^2 = \delta^2\|\nabla f(x) \|^2 \leq \sigma^2$ with $\delta$ sufficiently close to 1. With the update rule, we note that $x_{t+1} = x_t - \gamma_t$ w.p. 1/2 and $x_{t+1} = x_t + \gamma_t$ w.p. 1/2, and therefore
$$
\mathbb{E}_{\xi_t}\left[\|\nabla f(x_{t+1}) \vert x_t \in D\|\right] = \frac{1}{2}[L\|x_t -\gamma_t\| + L\|x_t + \gamma_t\|] \geq L\|x_t\| = \|\nabla f(x_t)\|.
$$
For $x \notin D$, we have $\|x\| > \epsilon/L + \gamma_{\max}$, and we assume there is no noise in the gradients. Therefore, if $x_t \notin D$, we know that after one step of update $\|x_{t+1}\| > \epsilon/L$, which implies $\|\nabla f(x_{t+1})\| > \epsilon$. Combining two cases that $x_t \in D$ and $x_t \notin D$, we know that $\mathbb{E}\|\nabla f(x_t)\| > \epsilon$ for all $t$.

\end{proof}

\paragraph{Proof of Proposition~\ref{thm:NSGD}}
\mbox{}\\
\begin{proof}
    Denote $ e_t = g(x_t; \xi_t) - \nabla f(x_t)$.
    By Lemma~2 in \citep{cutkosky2020momentum}, 
    \begin{equation*}
    f(x_{t+1}) - f(x_t) \leq - \frac{\gamma_t}{3}\|\nabla f(x_t)\| + \frac{8\gamma_t}{3}\| e_t || + \frac{\ell\gamma_t^2}{2}.
    \end{equation*}
    Telescoping from $t=0$ to $T-1$, 
    $$
    \frac{\gamma}{3 T^{\nicefrac{1}{2}}} \sum_{t = 0}^{T-1} \|\nabla f(x_t)\| \leq \frac{1}{3}\sum_{t = 0}^{T-1}  \gamma_t\|\nabla f(x_t)\|  \leq \Delta + \frac{8}{3 }\sum_{t = 0}^{T-1} \gamma_t  \|e_t\|  + \sum_{t = 0}^{T-1}\frac{\ell\gamma_t^2}{2} ,
    $$
    Taking expectation, rearranging and using $\Ep{ \norm{ e_t } } \leq \left( \Ep{\norm{e_t}^2} \right)^{1/2}\leq \sigma$, we derive
    \begin{align} \nonumber
    \sum_{t = 0}^{T-1} \Ep{ \|\nabla f(x_t)\|}  &\leq 3 T^{\nicefrac{1}{2} }\left[ \frac{\Delta}{\gamma} + \frac{8 \sigma }{3 } \sum_{t = 0}^{T-1} \frac{1}{(t+1)^{1/2}}  + \frac{\ell \gamma }{2} \sum_{t=0}^{T-1} \frac{1}{t+1}  \right] \\
    &\leq 3 T^{1/2}\left[ \frac{\Delta}{\gamma} + 8\sigma T^{1/2} + \ell \gamma \log(T)  \right]. \nonumber
    \end{align}
\end{proof}

\bigskip

\paragraph{Proof for Proposition \ref{thm:nsgdm}}
\mbox{}\\
\begin{proof}
We define $\hat e_t = {g}_t - \nabla f(x_t)$. By Lemma~2 in \citep{cutkosky2020momentum}, for any $\gamma_t > 0$
\begin{equation}\label{eq:nsgdm-descent}
    f(x_{t+1}) - f(x_t) \leq - \frac{\gamma_t}{3}\|\nabla f(x_t)\| + \frac{8\gamma_t}{3}\|\hat{e}_t|| + \frac{\ell\gamma_t^2}{2}.
\end{equation}
 Telescoping from $t = 0$ to $T - 1$,
 $$
 \frac{\gamma}{3 T^{\nicefrac{3}{4}}} \sum_{t = 0}^{T-1} \|\nabla f(x_t)\| \leq \frac{1}{3}\sum_{t = 0}^{T-1}  \gamma_t\|\nabla f(x_t)\|  \leq \Delta + \frac{8}{3}\sum_{t = 0}^{T-1}\gamma_t \|\hat{e}_t|| + \sum_{t = 0}^{T-1}\frac{\ell\gamma_t^2}{2} ,
 $$
 By taking expectation on both sides, rearranging and controlling the variance term using Lemma~\ref{le:hat_et_bound}, we derive
 \begin{align} \nonumber
 \sum_{t = 0}^{T-1} \Ep{ \|\nabla f(x_t)\| } &\leq 3T^{\nicefrac{3}{4} }\left[ \frac{\Delta}{\gamma} + \frac{8}{3 \gamma }\sum_{t = 0}^{T-1} \gamma_t \Ep{ \|\hat{e}_t|| } + \frac{\ell \gamma }{2} \sum_{t=0}^{T-1}(t+1)^{-3/2}  \right] \notag \\
 &\leq 3 T^{3/4}\left[ \frac{\Delta}{\gamma} + \frac{8}{3}\left(  C_1 \sigma  + C_2 \ell \gamma \right) \log(T)  +  \frac{ 2 \ell \gamma }{T^{1/2}}  \right] \notag \\
 &\leq C T^{3/4}\left[ \frac{\Delta}{\gamma} +  \left( \sigma  +  \ell \gamma \right) \log(T)  \right]. \nonumber
  \end{align}

\end{proof}

\begin{lemma}\label{le:hat_et_bound}
Under the setting of Theorem~\ref{thm:nsgdm}, there exist numerical constants $C_1, C_2 > 0$ such that for all $t \geq 1$,
$$
 \Ep{\norm{ \hat{e}_{t} }} \leq C_1 \sigma \alpha_t^{\nicefrac{1}{2}} + C_2 \ell \gamma_t \alpha_t^{-1}, 
$$
$$
\sum_{t=0}^{T-1} \gamma_t \Ep{ \norm{  \hat{e}_{t} }} \leq  \left( C_1 \sigma \gamma  + C_2 \ell \gamma^2 \right) \log(T)  ,
$$
where $\hat e_t = {g}_t - \nabla f(x_t)$. 
\end{lemma}
\begin{proof}
Define $e_t = g(x_t;\xi_t) - \nabla f(x_t)$, $S_t = \nabla f(x_t) - \nabla f(x_{t+1})$. Then 
\begin{eqnarray}
  \hat{e}_{t+1} &=& {g}_{t+1} - \nabla f(x_{t+1}) \notag \\
  & = & (1- \alpha_{t}) {g}_{t} + \alpha_{t} g(x_{t+1}; \xi_{t+1}) - \nabla f(x_{t+1}) \notag \\
  &=& (1- \alpha_t)  \hat{e}_t + \alpha_t  \epsilon_{t+1} + (1- \alpha_t) S_t  \notag .
\end{eqnarray}
Unrolling the recursion from $t = T-1$ to $t = 0$, we have
    \begin{eqnarray}\label{eq:hat_et_unrolled}
    \hat{e}_{T} &= &  \left( \prod_{t = 0}^{T-1} (1 - \alpha_{t}) \right) \hat{e}_0 + \sum_{t = 0}^{T-1} \alpha_{t} e_{t+1} \prod_{\tau = t + 1 }^{T-1} (1 - \alpha_{\tau}) + \sum_{t = 0}^{T-1} (1-\alpha_{t}) S_{t} \prod_{\tau = t + 1}^{T-1} (1 - \alpha_{\tau})\,. 
\end{eqnarray}
Define the $\sigma$-field $\mathcal{F}_{t} := \sigma( \left\{x_0, \xi_0, \dots, \xi_{t-1} \right\} )$. Notice that for any~$t_2 > t_1 \geq 0$  we have 

\begin{equation}\label{eq:et1_et2_independence}
\Ep{ \langle e_{t_1}, e_{t_2} \rangle }  = \Ep{ \Ep{ \langle e_{t_1}, e_{t_2} \rangle | \mathcal{F}_{t_2} } } = \Ep{ \langle e_{t_1}, \Ep{  e_{t_2}  | \mathcal{F}_{t_2} }  \rangle  } = 0.
\end{equation}
Then taking norm, applying expectation on both sides of \eqref{eq:hat_et_unrolled} and using $\Ep{\|\hat e_0\|} \leq \sigma $, we have
\begin{eqnarray}
    \Ep{\norm{ \hat{e}_{T} }} & \leq  &  \left( \prod_{t = 0}^{T-1} (1 - \alpha_{t}) \right) \sigma + \Ep{ \left\| \sum_{t = 0}^{T-1} \alpha_{t} e_{t+1} \prod_{\tau = t + 1 }^{T-1} (1 - \alpha_{\tau}) \right\| } \notag\\
    &&\qquad + \Ep{ \left\| \sum_{t = 0}^{T-1} (1-\alpha_{t}) S_{t} \prod_{\tau = t + 1}^{T-1} (1 - \alpha_{\tau})  \right\| } \,   \notag \\
    & \leq &  \left( \prod_{t = 0}^{T-1} (1 - \alpha_{t}) \right) \sigma + \left(\Ep{ \left\| \sum_{t = 0}^{T-1} \alpha_{t} e_{t+1} \prod_{\tau = t + 1 }^{T-1} (1 - \alpha_{\tau}) \right\|^2 } \right)^{\nicefrac{1}{2}} \notag \\
    && \qquad +  \sum_{t = 0}^{T-1} (1-\alpha_{t}) \Ep{ \left\| S_{t} \right\| } \prod_{\tau = t + 1}^{T-1} (1 - \alpha_{\tau})   \,   \notag \\
    & \leq &  \left( \prod_{t = 0}^{T-1} (1 - \alpha_{t}) \right) \sigma + \left( \sum_{t = 0}^{T-1} \alpha_{t}^2 \Ep{ \left\|e_{t+1}\right\|^2} \prod_{\tau = t + 1 }^{T-1} (1 - \alpha_{\tau})^2  \right)^{\nicefrac{1}{2}} \notag \\
    &&\qquad +  \ell  \sum_{t = 0}^{T-1} (1-\alpha_{t}) \gamma_t \prod_{\tau = t + 1}^{T-1} (1 - \alpha_{\tau})    \,   \notag \\
    & \leq &  \left( \prod_{t = 0}^{T-1} (1 - \alpha_{t}) \right) \sigma +  \left( \sum_{t = 0}^{T-1} \alpha_{t}^2  \prod_{\tau = t + 1 }^{T-1} (1 - \alpha_{\tau})  \right)^{\nicefrac{1}{2}} \sigma +  \left( \sum_{t = 0}^{T-1} \gamma_t \prod_{\tau = t + 1}^{T-1} (1 - \alpha_{\tau}) \right) \ell   \notag \, ,
\end{eqnarray}
where the first inequality holds by Jensen's inequality applied to $x \mapsto x^2$, the second inequality follows by \eqref{eq:et1_et2_independence} and the bound $\norm{S_t} \leq \ell \norm{x_{t+1}- x_t} = \ell \gamma_t$. The last step is due to bounded variance $\Ep{\|\hat \epsilon_0\|} \leq \sigma $ and $\alpha_t \leq 1$.

By the choice of momentum sequence, we have $\alpha_0=1$ and the first term is zero. By Lemma~\ref{le:sum_prod_bound1}, there exist numerical constants $C_1, C_2 > 0$ such that 
$$
\left( \sum_{t = 0}^{T-1} \alpha_{t}^2  \prod_{\tau = t + 1 }^{T-1} (1 - \alpha_{\tau})  \right)^{\nicefrac{1}{2}} \leq C_1 \alpha_T^{\nicefrac{1}{2}}, \qquad \left( \sum_{t = 0}^{T-1} \gamma_t \prod_{\tau = t + 1}^{T-1} (1 - \alpha_{\tau}) \right) \leq C_2 \gamma_T \alpha_T^{-1}.
$$ 
Therefore, for all $T \geq 1$, we have 
$$
 \Ep{\norm{ \hat{e}_{T} }} \leq  C_1 \sigma \alpha_T^{\nicefrac{1}{2}} + C_2 \ell \gamma_T \alpha_T^{-1}.
$$

\begin{eqnarray}
 \sum_{t=0}^{T-1} \gamma_t \Ep{ \norm{  \hat{e}_{t} }} &\leq&  C_1 \sigma \sum_{t=0}^{T-1} \gamma_t \alpha_t^{\nicefrac{1}{2}} + C_2 \ell \sum_{t=0}^{T-1} \gamma_t^2 \alpha_t^{-1} \notag \\
 &\leq&  C_1 \sigma \gamma \sum_{t=0}^{T-1} (t+1)^{-\nicefrac{3}{4}} (t+1)^{-\nicefrac{1}{4}} + C_2 \ell \gamma^2 \sum_{t=0}^{T-1} (t+1)^{-\nicefrac{3}{2}} (t+1)^{\nicefrac{1}{2}}\notag \\
 &\leq& \left( C_1 \sigma \gamma  + C_2 \ell \gamma^2 \right) \log(T) \notag .
\end{eqnarray}

\end{proof}

\begin{lemma}[Lemma 15 in \citep{Fatkhullin_SPG_FND_2023}]\label{le:sum_prod_bound1}
Let $q \in [0, 1)$, $p \geq 0$, $\gamma_0 > 0$ and let $\eta_t =  \rbr{ \frac{2}{t+2} }^q$, $\gamma_t = \gamma_0 \rbr{ \frac{1}{t+1} }^p$ for every integer~$t$. Then for any integers~$t$ and~$T \geq 1$, it holds 
    $$
     \sum_{t = 0}^{T-1} \gamma_{t} \prod_{\tau = t+1 }^{T-1} (1 - \eta_{\tau}) \leq C \gamma_t \eta_T^{-1},
    $$
   where $C := 2^{p-q}(1-q)^{-1} t_0 \exp \left(2^q(1-q) t_0^{1-q}\right)+2^{2 p+1-q}(1-q)^{-2}$ and $t_0 := \max \left\{\left(\frac{p}{(1-q) 2^q}\right)^{\frac{1}{1-q}}, 2\left(\frac{p-q}{(1-q)^2}\right)\right\}^{\frac{1}{1-q}}$.
\end{lemma}

\section{Proofs for Scalar AMSGrad and AdaGrad in Section \ref{sec:adaptive}}
\label{sec:apdx_adaptive_amsgrad}

\bigskip

The following is an extended version of Theorem \ref{theorem:ams_grad} including $\gamma_t=\frac{\gamma}{(t+1)^{\alpha}}$ with  $0 < \alpha < 1$. 

\begin{theorem}
\label{thm:amsgrad_general_ss}
    Under Assumption \ref{assume:smoothness},
    if we run AMSGrad-norm with $\gamma_t=\frac{\gamma}{(t+1)^{\alpha}}$, $v_0 > 0$ and $\beta_1=\beta_2=0$ in the deterministic setting, then
    for any $\gamma > 0$ and $0 < \alpha < 1$, if $v_0 < \gamma \ell$
    \[
        \frac{1}{T} \sum_{t=0}^{T-1} \|\nabla f(x_t)\|^2 
         \leq  \frac{2 \Delta }{\gamma T^{1-\alpha}}  \max\{v_0, \sqrt{2\ell\Delta} \} ,
    \] 
    if $v_0 \geq \gamma \ell$
    \[
        \frac{1}{T} \sum_{t=0}^{T-1} \|\nabla f(x_t)\|^2 
        \leq \frac{\left( \frac{\ell \gamma}{v_{0}}\right)^{\frac{1}{\alpha}} \gamma^2 \ell^2 }{T} + \frac{2(M + \Delta)  }{\gamma T^{1-\alpha}} \max\{\gamma\ell, \sqrt{2\ell(M+\Delta)} \} ,
    \] 
    where 
    \begin{align*}
    M
 = \left\{
\begin{aligned}
    & \ell\gamma^2 \left( 1+ \log \left( \frac{\ell\gamma}{v_0}\right)\right), \text{ when } \alpha = 1/2, \\
    &\frac{\ell\gamma^2}{2\left(1-2^{1-2\alpha} \right)}, \text{ when } 1/2 < \alpha  < 1, \\
    &\frac{\gamma(\ell\gamma)^{\frac{1}{\alpha} - 1}}{2(1-2\alpha)v_0^{\frac{1}{\alpha}-2}}, \text{ when } 0 < \alpha  < 1/2.
\end{aligned}
\right.
\end{align*}
\end{theorem}

\begin{proof}
The effective stepsize of AMSGrad-norm contains a maximum over all gradient
norms in the denominator. As it is desirable to find a lower bound for the
effective stepsize, we begin by bounding the gradient norms.

Let $\tau$ be the first iteration where the effective stepsize is less or equal
to $1/\ell$, i.e., $\eta_{\tau - 1} > 1/\ell$ and $\eta_\tau \leq 1 / \ell$. First, we assume $\tau \geq 1$, i.e., $v_0 < \gamma \ell$.
The time stamp $\tau$ itself is naturally bounded by
\[
    \eta_{\tau - 1} = \frac{\gamma}{\tau^\alpha v_\tau} > \frac{1}{\ell} 
    \quad \Longrightarrow \quad 
    \tau < \left( \frac{\ell \gamma}{v_{\tau}}\right)^{\frac{1}{\alpha}}
    \leq \left( \frac{\ell \gamma}{v_{0}}\right)^{\frac{1}{\alpha}}.
\]
 We have
\begin{equation} \label{eq:amsgrad_befre_tau}
        \sum_{t=0}^{\tau -1} \| \nabla f(x_t) \|^2 \leq 
    \tau \gamma^2 \ell^2 \leq \left( \frac{\ell \gamma}{v_{0}}\right)^{\frac{1}{\alpha}} \gamma^2 \ell^2.
\end{equation}
By $\ell$-smoothness of $f(\cdot)$,
\begin{align} \nonumber
    f(x_{t+1}) & \leq f(x_t) + \langle \nabla f(x_t), x_{t+1} - x_t \rangle + \frac{\ell}{2}\|x_{t+1}- x_t \|^2 \\ \label{eq:step_decrease}
    & = f(x_t) - \eta_t \| \nabla f(x_t)\|^2 + \frac{\ell\eta_t^2}{2} \|\nabla f(x_t)\|^2 \\ \nonumber
    &\leq f(x_t) + \frac{\ell\eta_t^2}{2} \|\nabla f(x_t)\|^2.
\end{align}
Therefore,
\begin{align*}
    f(x_{\tau}) - f(x_0) & \leq \frac{\ell}{2}\sum_{t=0}^{\tau - 1}\eta_t^2\| \nabla f(x_t) \|^2  = \frac{\ell}{2}\sum_{t=0}^{\tau-1} \frac{\gamma_t^2}{v_{t+1}^2} \| \nabla f(x_t) \|^2 
\leq \frac{\ell}{2}\sum_{t=0}^{\tau-1} \gamma_t^2 \\
& \leq \left\{
\begin{aligned}
    &\frac{\ell\gamma^2}{2} (1 + \log \tau), \text{ when } \alpha = 1/2, \\
    &\frac{\ell\gamma^2}{2\left(1-2^{1-2\alpha} \right)}, \text{ when } 1/2 < \alpha  < 1, \\
    &\frac{\ell\gamma^2\tau^{1-2\alpha}}{2(1-2\alpha)}, \text{ when } 0 < \alpha  < 1/2.
\end{aligned}
\right.
\end{align*}
We denote the right hand side as $M$. Also from  (\ref{eq:step_decrease}) and definition of $\tau$, we know that $f(x_t) \leq f(x_{\tau})$ for $t \geq \tau$ and therefore, for all $t \geq \tau$,
$$
f(x_t) - f^* = f(x_{\tau}) - f(x_0) + f(x_0) - f^* \leq M + \Delta,
$$
which implies
$$
\|\nabla f(x_t)\|^2 \leq 2\ell(f(x_t) - f^*) \leq 2\ell(M + \Delta).
$$
Therefore, we can bound for all $t \geq 0$, 
$$v_t \leq \max\{\gamma\ell, \sqrt{2\ell(M+\Delta)} \}. $$ 
For $t \geq \tau$, by (\ref{eq:step_decrease})
\begin{align*}
    f(x_{t+1}) - f(x_t) 
     \leq - \frac{\eta_t}{2} \| \nabla f(x_t) \|^2.
\end{align*}
By telescoping from $t=\tau$ to $T-1$, we get
\begin{align*}
    2 \left(f(x_\tau) - f(x_T)\right)
    &\geq \sum_{t=\tau}^{T-1} \eta_t \| \nabla f(x_t) \|^2 \\
    &= \sum_{t=\tau}^{T-1} \frac{\gamma}{(t+1)^{\alpha} v_{t+1}} \| \nabla f(x_t) \|^2 \\
    &\geq \sum_{t=\tau}^{T-1} \frac{\gamma}{T^{\alpha} v_{t+1}} \| \nabla f(x_t) \|^2 \\
    &\geq \sum_{t=\tau}^{T-1} \frac{\gamma}{T^{\alpha} \max\{\gamma\ell, \sqrt{2\ell(M+\Delta)} \}} \| \nabla f(x_t) \|^2.
\end{align*}
Then we have
\begin{align*}
    \sum_{t=\tau}^{T-1} \| \nabla f(x_t) \|^2
    &\leq \frac{2}{\gamma} \left(f(x_\tau) - f(x_T)\right) T^{\alpha} \max\{\gamma\ell, \sqrt{2\ell(M+\Delta)} \} \\
    &\leq \frac{2}{\gamma} \left(f(x_\tau) - f(x^*)\right) T^{\alpha} \max\{\gamma\ell, \sqrt{2\ell(M+\Delta)} \}\\
    &\leq \frac{2(M + \Delta)}{\gamma} T^{\alpha} \max\{\gamma\ell, \sqrt{2\ell(M+\Delta)} \}.
\end{align*}
Combining with (\ref{eq:amsgrad_befre_tau}),
we obtain 
\begin{equation*}
    \sum_{t = 0}^{T-1} \| \nabla f(x_t) \|^2 \leq  \left( \frac{\ell \gamma}{v_{0}}\right)^{\frac{1}{\alpha}} \gamma^2 \ell^2 + \frac{2(M + \Delta)  T^{\alpha}}{\gamma} \max\{\gamma\ell, \sqrt{2\ell(M+\Delta)} \}.
\end{equation*}
When $\tau = 0$, we have
\begin{align*}
    2 \left(f(x_0) - f(x_T)\right)
    &\geq
     \sum_{t=0}^{T-1} \frac{\gamma}{T^{\alpha} v_{t+1}} \| \nabla f(x_t) \|^2 \geq \sum_{t=\tau}^{T-1} \frac{\gamma}{T^{\alpha} \max\{v_0, \sqrt{2\ell\Delta} \}} \| \nabla f(x_t) \|^2,
\end{align*}
which implies
\begin{equation*}
    \sum_{t = 0}^{T-1} \| \nabla f(x_t) \|^2 \leq  \frac{2 \Delta T^{\alpha}}{\gamma}  \max\{v_0, \sqrt{2\ell\Delta} \}.
\end{equation*}

\begin{remark}
    For any $0 < \alpha < 1$, if we compare simplified AMSGrad with $\gamma_t = \frac{\gamma}{(t+1)^{\alpha}}$ to SGD with $\eta_t = \frac{\eta}{(t+1)^{\alpha}}$ in the deterministic case (setting $\sigma = 0$ in Theorem \ref{thm:general_stepsize}), we observe that they achieve the same convergence rate. However, the complexity of simplified AMSGrad only includes polynomial term in $\gamma$ and $\ell$, while that of SGD includes an exponential term in $(\eta\ell)^{1/\alpha}$.
\end{remark}

\end{proof}

In the following, we will first provide the lower bounds for scalar version of AMSGrad (referred to as AMSGrad-norm) with each $\alpha \in (0, 1)$ and  discuss why it may fail with $\alpha = 0$ when problem parameters are unknown, which means that it can not achieve the optimal complexity $\cO(\epsilon^{-2})$ in the deterministic setting. Second, we show that it also fails to achieve the optimal convergence rate in the stochastic setting when stochastic gradients are unbounded.
To make the results more general,
we consider the standard scalar AMSGrad with momentum hyper-parameters $\beta_1$ and $\beta_2$, which is presented in Algorithm~\ref{alg:amsgrad_full}.

Before proceeding to our results, we present a lemma which is handy for
conducing lower bounds for SGD-like algorithms with momentum (see Algorithm~\ref{alg:sgd_momentum}).
As long as an upper bound is known for stepsize $\eta_t$, we can derive a lower bound
similar to Proposition~1 in \citep{drori2020complexity}.

\begin{algorithm}[ht] 
    \caption{General SGD with Momentum}
    \begin{algorithmic}[1]
      \STATE \textbf{Input:} initial point $x_0$, momentum parameters $0 \leq \beta_1 < 1$ and
      initial moment $m_0$.
        \FOR{$t = 0,1,2,...$}
            \STATE sample  $\xi_{t}$
            \STATE $m_{t+1} = \beta_1 m_t + (1 - \beta_1) g(x_t; \xi_t)$
            \STATE set a stepsize $\eta_t > 0$
            \STATE  $x_{t+1} = x_t - \eta_t m_{t+1}$
        \ENDFOR
    \end{algorithmic} \label{alg:sgd_momentum}
\end{algorithm}

\begin{lemma} \label{lemma:lower_sgd_momentum}
    For any $\ell > 0$, $\Delta > 0$ and $T > 1$, there exists a
    $\ell$-smooth function $f: \mathbb{R} \rightarrow \mathbb{R}$, and $x_0$ with $f(x_0) - \inf_x f(x) \leq \Delta$, such that if we
    run Algorithm~\ref{alg:sgd_momentum} with deterministic gradients and $\eta_t \leq \tilde \eta_t$ for $t=0, 1, 2,..., T-1$,
    then we have
    \[
        \min_{t \in \{0, 1, ..., T-1\}} |\nabla f(x_t) | \geq \sqrt{\frac{\Delta}{16 \max\{1/\ell, \sum_{t=0}^{T-1} \tilde \eta_t\}}}.
    \]
\end{lemma}

\begin{proof}
    We construct a quadratic function similar to Proposition~1 in \citep{drori2020complexity}.
    The following function is considered:
    \[
        f(x) = \frac{x^2}{4\max\left\{1/\ell, \sum_{t=0}^{T-1} \tilde \eta_t \right\}}.
    \]
    Without loss of generality, we assume the initial moment $m_0$ is non-positive, and we set the initial point $x_0$ as
    \[
        x_0 = \sqrt{\Delta \max\left\{1/\ell, \sum_{t=0}^{T-1} \tilde \eta_t \right\}}.
    \]
    Otherwise if the initial moment is set to be positive, then we let $x_0$ be negative and follow the same reasoning.

    Since $x_0$ is positive, the first gradient direction would be positive, i.e., $\nabla f(x_0) > 0$. Let $\tau$ be the first iteration such that $m_{\tau} > 0$. By the update rule and definition of $\tau$, it is obvious that $x_t \geq x_0$ for $t \leq \tau -1$. 
    If $T \leq \tau$, it trivially holds that $\nabla f(x_t) \geq \nabla f(x_0)$ for all $0 \leq t \leq T - 1$. Otherwise, 
    we have $m_{\tau} = \beta_1 m_{\tau - 1} + (1 - \beta_1)\nabla f(x_{\tau - 1}) \leq (1 - \beta_1)\nabla f(x_{\tau - 1})$. That is to say, the gradient estimation $m_\tau$ used in the
    $\tau$-th step has the correct direction but its magnitude is no larger than
    the actual gradient. Starting from the $\tau$-th iteration, $x_t$ will monotonically 
    move left towards the solution.
    Note that since our stepsize is small enough, i.e.,
    \[
    \eta_t \leq \tilde \eta_t < 2\max\left\{1/\ell, \sum_{t=0}^{T-1} \tilde \eta_t \right\},
    \]
    the updates will remain positive, i.e., $x_t > 0$ for $t \geq \tau$. By the update rule, we note that $x_{t+1} \leq x_t$ for $t \geq \tau$, and therefore $\nabla f(x_{t+1}) < \nabla f(x_t)$. 
    We can conclude that for any $t \geq \tau$, we have $m_t \leq \nabla f(x_{\tau-1})$.
    Then for $t \geq \tau - 1$ we have
    \begin{align*}
        x_t & = x_{\tau - 1} - \sum_{k=\tau - 1}^{t-1} \eta_t m_{t+1} \\
            &\geq x_{\tau - 1} - \sum_{k=\tau - 1}^{t-1} \tilde \eta_t \nabla f(x_{\tau-1}) \\
            &= x_{\tau - 1} - \sum_{k=\tau - 1}^{t-1} \frac{\tilde \eta_t}{2\max\left\{1/\ell, \sum_{t=0}^{T-1} \tilde \eta_t \right\}}x_{\tau-1} \\
            &\geq \frac{1}{2} x_{\tau - 1} \\
            &\geq \frac{1}{2} x_0.
    \end{align*}
    Then we conclude by 
    \begin{align*}
    |\nabla f(x_t)| = \frac{x_t}{2\max\left\{1/\ell, \sum_{t=0}^{T-1} \tilde \eta_t \right\}}
    \geq \frac{x_0}{4\max\left\{1/\ell, \sum_{t=0}^{T-1} \tilde \eta_t \right\}}
    =\sqrt{\frac{\Delta}{16 \max\{1/\ell, \sum_{t=0}^{T-1} \tilde \eta_t\}}}.
    \end{align*}
\end{proof}

Now we proceed to provide the lower bound for deterministic case.

\begin{theorem}
\label{thm:amsgrad_deter_lower_bound}
    For any $\ell > 0$, $\Delta > 0$ and $T > 1$, there exists a
    $\ell$-smooth function $f: \mathbb{R} \rightarrow \mathbb{R}$ and $x_0$ with $f(x_0) - \inf_x f(x) \leq \Delta$, such that if we
    run Algorithm~\ref{alg:amsgrad_full} with deterministic gradients, $0 < v_0 \leq \frac{\ell \gamma}{2}$,
    and $\gamma_t = \frac{\gamma}{{(t+1)}^{\alpha}}$ with $\gamma \leq \frac{4\Delta}{v_0}$, we have
    (1) if $0 < \alpha < 1$, for any $0 \leq \beta_1 < 1$
    and $0 \leq \beta_2 \leq 1$, we have
    \[
        \min_{t \in \{0, 1, ..., T-1\}} |\nabla f(x_t) | \geq \sqrt{\frac{\Delta}{16 \max\{1/\ell, \frac{\gamma}{(1-\alpha) v_0} T^{1-\alpha} \}}},
    \]
    and (2) if $\alpha = 0$, for $\beta_1=0$ and any $0 \leq \beta_2 \leq 1$, we have
    \[
        \min_{t \in \{0, 1, ..., T-1\}} |\nabla f(x_t) | \geq v_0.
    \]
\end{theorem}

\begin{remark}
    From the theorem, we can conclude that the optimal convergence rate $\frac{1}{\sqrt{T}}$
    for $\|\nabla f(x_t)\|$ is infeasible for AMSGrad with polynomially decreasing stepsize.
    When $\alpha=0$, a similar result can be obtained 
    for the case $\beta_1 \geq 0$, $\beta_2 = 0$ and small enough $v_0$.
\end{remark}
\begin{proof}
    For $\alpha > 0$, we have
    \[
        \eta_t = \frac{\gamma}{{(t+1)}^{\alpha} \sqrt{\hat v^2_{t+1}}}
        \leq \frac{\gamma}{{(t+1)}^{\alpha} v_0}.
    \]
    Let $\tilde \eta_t = \frac{\gamma}{{(t+1)}^{\alpha} v_0}$ and then we have
    \[
        \sum_{t=0}^{T-1} \tilde \eta_t = \sum_{t=0}^{T-1} \frac{\gamma}{{(t+1)}^{\alpha} v_0} 
        \leq  \frac{\gamma}{(1-\alpha)v_0} T^{1-\alpha}.
    \]
    Applying Lemma~\ref{lemma:lower_sgd_momentum} directly gives us the desired result.
    
    For $\alpha = 0$, we consider function
    \[
        f(x) = \frac{v_0}{\gamma} x^2.
    \]
    Note that since $v_0 \leq \frac{\ell\gamma}{2}$, the function is $\ell$-smooth.
    Let 
    \[
    x_0 = \frac{\gamma}{2},
    \]
    which satisfies the condition that $f(x_0) \leq \Delta$. Then after one update
    \begin{align*}
        v^2_1 = \beta_2 v^2_0 + (1-\beta_2) \|\nabla f(x_0)\|^2 = v^2_0 \\
        x_1 = x_0 - \frac{\gamma}{\sqrt{v^2_1}} \nabla f(x_0) = - \frac{\gamma}{2} = -x_0.
    \end{align*}
    If we continue this calculation, we find that the iterates will oscillate between
    $\frac{\gamma}{2}$ and $-\frac{\gamma}{2}$ forever, which finishes the proof.
\end{proof}

\paragraph{Proof for Theorem \ref{thm:amsgrad_stoc}}
\mbox{}\\

\begin{proof}
    We consider a two-dimensional function $f: \mathbb{R}^2 \rightarrow \mathbb{R}^2$, for $x = (x^1, x^2)^\top \in \mathbb{R}^2$,
    $$
    f(x) =  F(x^1),
    $$
    where its function value only depends on the first dimension and we will define $F: \mathbb{R} \rightarrow \mathbb{R}$ later. The gradient at $x$ is $\nabla f(x) = (\nabla F(x^1), 0)^\top$. We add the noise
     only  to the second dimension, i.e., $g(x; \xi) = (\nabla F(x^1), \xi)$. %
    For any $t \geq 0$, the probability density function of the noise as
    \[
        p_{\xi_t}(x) = \begin{cases} \frac{1}{s\zeta} \left(\frac{x}{s}\right)^{-1-\frac{2}{\zeta}}
            e^{-\left(\frac{x}{s}\right)^{-\frac{2}{\zeta}}}, &x \geq 0; \\
            \frac{1}{s\zeta} \left(\frac{-x}{s}\right)^{-1-\frac{2}{\zeta}}
            e^{-\left(\frac{-x}{s}\right)^{-\frac{2}{\zeta}}}, &x < 0,
        \end{cases}
    \]
    where  $s = \frac{\sigma}{\sqrt{\Gamma\left(1-\frac{\zeta}{2}\right)}}$.
    Note that the distribution is symmetric and $\Ep{\xi_t} = 0$.
    Also, we note that $|\xi_t|$ follows the Fr\'echet distribution~\citep{de2011generalized} with
    cumulative distribution function
    \[
        \Pr (|\xi_t| \leq x) = e^{-\left(\frac{x}{s}\right)^{-\frac{2}{\zeta}}},
    \]
    and
    \begin{align*}
        \Var{\xi_t} &= \Ep{|\xi_t|^2} - (\Ep{\xi_t})^2 \\
                    &= s^2 \Gamma\left(1-\frac{\zeta}{2}\right) - (\Ep{\xi_t})^2 \\
                    &\leq s^2 \Gamma\left(1-\frac{\zeta}{2}\right)
                    \\ 
                    & \leq \sigma^2,
    \end{align*}
    where we used the exact second moment for Fr\'echet distribution.
    
    Next, we will show that $\tilde \xi_t := \max_{0 \leq k \leq t}\{|\xi_k|\} \geq \Omega\left(\frac{1}{(t+1)^{\zeta - 1/2}}\right)$ with probability $\frac{1}{2}$.
    We know that $\tilde \xi_t$ also follows Fr\'echet 
    distribution with CDF
    \[
        \Pr (\tilde \xi_t \leq x) = \exp\left( {-\left(\frac{x}{s \cdot (t+1)^{\frac{\zeta}{2}}}\right)^{-\frac{2}{\zeta}}} \right).
    \]
    Then for constant $C > 0$,
    \begin{align*}
        \Pr (\tilde \xi_t \leq C\cdot (t+1)^{\zeta - \frac{1}{2}}) 
        &= \exp\left({-\left(\frac{C\cdot (t+1)^{\zeta - \frac{1}{2}}}{s \cdot (t+1)^{\frac{\zeta}{2}}}\right)^{-\frac{2}{\zeta}}} \right)\\
        &= \exp\left({- \left(\frac{C}{s}\right)^{\frac{2}{\zeta}} (t+1)^{\frac{1}{\zeta} - 1}} \right) \\
        &\leq \frac{1}{4(t+1)^2},
    \end{align*}
    where the last inequality is by selecting $C = \frac{s\left(e\left(\frac{1}{\zeta} - 1 \right)\right)^{\frac{\zeta}{2}}}{\sqrt{2}}$
    and using $\exp\left({- \frac{2^{m+1}}{e m}\cdot t^{m}} \right) \leq \frac{1}{4t^2}$ for any $t>0$ and $0<m<1$.
    Then using union bound, we have
    \begin{align*}
        \Pr (\tilde \xi_t > C\cdot (t+1)^{\zeta - \frac{1}{2}} \quad \text{for}\quad 0 \leq t \leq T-1) 
        \geq  1 - \sum_{t=0}^{T-1} \frac{1}{4(t+1)^2} \geq \frac{1}{2}.
    \end{align*}

    Now we have shown that with some probability, the noise is large enough. We can use this property
    to provide an upper bound $\tilde \eta_t$ for the stepsize as follow
    \begin{align*}
        \eta_t &= \frac{\gamma}{\sqrt{t+1} \sqrt{\tilde v_{t+1}}} \\
               &= \frac{\gamma}{\sqrt{t+1} \sqrt{\max_{0\leq k\leq t} \{\beta_2 v_{k} + (1-\beta_2) \|g(x_k; \xi_k) \|^2\}}} \\
               &\leq \frac{\gamma}{\sqrt{t+1} \sqrt{\max_{0\leq k\leq t} \{(1-\beta_2) \|g(x_k; \xi_k) \|^2\}}} \\
               &\leq \frac{\gamma}{\sqrt{t+1} \sqrt{\max_{0\leq k\leq t} \{(1-\beta_2) \|\xi_k\|^2\}}} \\
               &= \frac{\gamma}{\sqrt{t+1} \sqrt{(1-\beta_2)} \tilde \xi_t } \\
               &\leq \frac{\gamma}{C (t+1)^{\zeta} \sqrt{(1-\beta_2)} } \triangleq \Tilde{\eta}_t,
    \end{align*}
    This implies
    \begin{align*}
        \sum_{t=0}^{T-1} \tilde \eta_t \leq \frac{\gamma}{(1-\zeta) C \sqrt{(1-\beta_2)} }\left(T^{1-\zeta} - \zeta \right).
    \end{align*}
    We observe that the update with AMSGrad-norm in function $f$ corresponds to applying general SGD with momentum (Algorithm \ref{alg:sgd_momentum}) to function $F$ with stepsize $\eta_t$. Therefore, we can pick a hard instance $F$ according to Lemma~\ref{lemma:lower_sgd_momentum}, and by noting that $\|\nabla f(x)\| = |\nabla F(x^1)|$ we reach our conclusion.

\end{proof}

\begin{remark}
As we see above, the function $F$ in the proof is constructed by Lemma~\ref{lemma:lower_sgd_momentum}. We note that even assuming the gradients of $f$ to be bounded, i.e., $\|\nabla f(x)\| \leq K$ for all $x$, will not prevent the slow convergence in Theorem \ref{thm:amsgrad_stoc}. This is  because in the proof of Lemma~\ref{lemma:lower_sgd_momentum} all iterates stay between $[0, x_{\tau-1}]$ (e.g., $\tau = 1$ if $m_0 =0$), so we can construct any Lipschitz function outside of this segment. 
    
\end{remark}

\bigskip

\begin{algorithm}[ht] 
    \caption{AdaGrad-norm}
    \setstretch{1.23}
    \begin{algorithmic}[1]
      \STATE \textbf{Input:} initial point $x_0$, $v_0 > 0$ and $\eta > 0$
        \FOR{$t = 0,1,2,...$}
            \STATE sample  $\xi_t$
            \STATE $v_{t+1}^2 = v_t^2 + \|g(x_k;\xi_k)\|^2$
            \STATE  $x_{t+1} = x_t - \frac{\eta }{\sqrt{v_{t+1}^2}}g(x_t; \xi_t)$
        \ENDFOR
    \end{algorithmic} \label{alg:adagrad}
\end{algorithm}

\paragraph{Proof for Proposition \ref{prop:adagrad}}
\mbox{}\\
\begin{proof}
Define a function $\widetilde{f}: \mathbb{R}^d \times \mathbb{R} \rightarrow \mathbb{R}$ such that $\widetilde{f}(x, y) = f(x) - \frac{\ell}{2}y^2$. Since the $\widetilde{f}$ is  $\ell$-smooth and $\ell$-strongly concave about $y$, the condition number is defined to be $\kappa = 1$. Applying AdaGrad-norm to $f$ is equivalent to applying NeAda-AdaGrad (Algorithm 3 in \citep{yang2022nest}) to $\widetilde{f}$ with $y_t \equiv 0$. For every $x$, we know $y^*(x) \triangleq \argmax_{y} \widetilde{f}(x, y) = 0 $. Then  $\mathcal{E} \triangleq \sum_{t=0}^{T-1} \frac{\ell^2\|y_t - y^*(x_t)\|^2 }{2v_0} = 0$. Plugging in $\kappa = 1$, $\mathcal{E} = 0$ and batchsize $M = 1$ to Theorem 3.1 in \citep{yang2022nest}, we reach the conclusion.

\end{proof}

\end{document}